\documentclass[12pt]{amsart}

\usepackage{amsmath}
\usepackage{amsfonts}
\usepackage{amssymb}
\usepackage{mathrsfs}
\usepackage{stmaryrd}

\usepackage{color}
\usepackage[colorlinks=true,pagebackref,hyperindex,citecolor=green,linkcolor=blue,urlcolor=blue]{hyperref}

\usepackage{lmodern} %

\theoremstyle{plain}
\newtheorem{theorem}{Theorem}
\newtheorem{lemma}[theorem]{Lemma}
\newtheorem{proposition}[theorem]{Proposition}
\newtheorem{corollary}[theorem]{Corollary}
\newtheorem{setup}[theorem]{Setup}

\theoremstyle{definition}
\newtheorem{definition}[theorem]{Definition}
\newtheorem{remark}[theorem]{Remark}
\newtheorem{example}[theorem]{Example}

\numberwithin{theorem}{section}
\numberwithin{equation}{section} %Can replace {subsection} with {theorem} if you want

\renewcommand{\(}{\left (}
\renewcommand{\)}{\right )}
\newcommand{\set}[1]{ \left \{ #1 \right \} }

\newcommand{\NN}{\mathbb{N}}
\newcommand{\ZZ}{\mathbb{Z}}
\newcommand{\QQ}{\mathbb{Q}}
\newcommand{\RR}{\mathbb{R}}
\newcommand{\FF}{\mathbb{F}}
\newcommand{\kk}{\Bbbk}

\newcommand{\m}{\mathfrak{m}}
\newcommand{\n}{\mathfrak{n}}

\newcommand{\cone}{\operatorname{cone}}
\newcommand{\vol}{\operatorname{vol}}
\newcommand{\T}{T_{\bullet}}
\newcommand{\U}{U_{\bullet}}
\newcommand{\interior}[1]{{#1}^{\circ}}
\newcommand{\boundary}[1]{\partial{#1}}
\newcommand{\closure}[1]{\overline{#1}}

\newcommand{\vv}[1]{\mathbf{#1}} % Stolen from Pedro, \vv = vector
\newcommand{\vs}[1]{#1} 
\newcommand{\iprod}[2]{ \langle {#1}, {#2} \rangle}

\newcommand{\size}[1]{\#{#1}}

\newcommand{\idealm}{\mathfrak{m}}
\newcommand{\idealp}{\mathfrak{p}}
\newcommand{\idealb}{\mathfrak{b}}

\newcommand*{\longhookrightarrow}{\ensuremath{\lhook\joinrel\joinrel\relbar\joinrel\rightarrow}}

\renewcommand*{\longrightarrow}{\ensuremath{\relbar\joinrel\relbar\joinrel\rightarrow}}

\newcommand*{\longonto}{\ensuremath{\relbar\joinrel\twoheadrightarrow}}

\usepackage{setspace}
\begin{document}
\title{Local Okounkov bodies and limits in prime characteristic}
\author{Daniel J. Hern\'andez and Jack Jeffries}
\thanks{Hern\'andez was supported in part by the NSF Postdoctoral Research Fellowship DMS \#1304250 and NSF grant DMS \#1600702.}
\thanks{Jeffries was supported in part by the NSF Postdoctoral Research Fellowship DMS \#1606353}
\maketitle
%\tableofcontents

\begin{abstract}
This article is concerned with the asymptotic behavior of certain sequences of ideals in rings of prime characteristic.  These sequences, which we call $p$-families of ideals,  are ubiquitous in prime characteristic commutative algebra (e.g., they occur naturally in the theories of tight closure, Hilbert-Kunz multiplicity, and $F$-signature). We associate to each $p$-family of ideals an object in Euclidean space that is analogous to the Newton-Okounkov body of a graded family of ideals, which we call a $p$-body.  Generalizing the methods used to establish volume formulas for the Hilbert-Kunz multiplicity and $F$-signature of semigroup rings, we relate the volume of a $p$-body to a certain asymptotic invariant determined by the corresponding $p$-family of ideals.  We apply these methods to obtain new existence results for limits in positive characteristic, an analogue of the Brunn-Minkowski theorem for Hilbert-Kunz multiplicity, and a uniformity result concerning the positivity of a $p$-family. 
\end{abstract}

\section{Introduction}

%\!{Maybe add something about how $q=p^e$, and simplify the notation in this section}

This article is concerned with limits of certain sequences indexed by the powers of a prime integer $p>0$, typically the characteristic of some fixed ambient ring.  To simplify notation, $q$ will often denote an integral power of this prime integer;  explicitly, $q=p^e$ for some nonnegative integer $e$.  Thus, a sequence of rational numbers $\{ \lambda_{p^e} \}_{e=0}^{\infty}$ is often written as $\{ \lambda_{q} \}_{q=1}^{\infty}$, and the limit $\lim \limits_{e \to \infty} \lambda_{p^e}$ as $\lim \limits_{q \to \infty} \lambda_q$.  

\subsection{Numerical limits in prime characteristic}\label{limitscharp}
Let $(R, \idealm)$ be a reduced local ring of prime characteristic $p>0$ and dimension~$d$, and let $\ell_R(M) = \ell(M)$ denote the length of an $R$-module $M$.   If $q$ is a power of $p$, then $R^q$ is the subring of $R$ consisting of $q$-th powers.  

In \cite{Kunz69}, Kunz showed that the following conditions are equivalent:
\begin{enumerate}
\item $R$ is a regular local ring.
\item $\ell_R(R/\idealm^{[q]}) = q^d$ for some (equivalently, every) $q$ a power of $p$.
\item $R$ is flat over $R^q$ for some (equivalently, every) $q$ a power of $p$.
\end{enumerate}

This theorem has inspired many numerical approaches to measuring the failure of a ring of positive characteristic to be regular. One of the first such avenues was initiated by Kunz \cite{Kunz76}, who investigated the function  $q \mapsto \ell_R(R/I^{[q]})$, where $q$ is a power of $p$ and $I$ is some fixed $\m$-primary ideal of $R$, now called the \emph{Hilbert-Kunz function} of $I$. While the definition of this function is analogous to that of the Hilbert-Samuel function of $I$ given by $t \mapsto \ell_R(R/I^{t})$, which agrees with a polynomial for large values of $t$, the Hilbert-Kunz function exhibits much more subtle behavior \cite{HanMonsky93}.

Nonetheless, Monsky showed in \cite{Monsky83} that the limit
\begin{equation}\label{HK}
 e_{HK}(I,R) := \lim_{e \rightarrow \infty} \frac{\ell_R(R/I^{[q]})}{q^d} 
\end{equation}
exists for any $\idealm$-primary ideal $I$. As with the Hilbert-Kunz function, this \emph{Hilbert-Kunz multiplicity} is a more subtle analogue of its characteristic-free counterpart; e.g., it may take irrational values \cite{Brenner14}.

Another similar approach to quantifying singularity using the Frobenius map is given by counting for each $q$ a power of the characteristic, the maximal rank $a_q(R)$ of an $R^{q}$-free summand of $R$. Under the simplifying assumption that the residue field of $R$ is perfect,  Kunz' theorem tells us that $a_q(R)$ is equal to the $q^d$, the rank of $R$ as an $R^{q}$-module,  if and only if $R$ is regular.

The limit
\begin{equation}\label{signature} s(R) := \lim_{ q \rightarrow \infty} \frac{a_q(R)}{q^d} 
\end{equation}
is studied in \cite{SmithVDB97}, and is defined in \cite{HunekeLeuschke02} as the \emph{$F$-signature} of $R$, provided the limit exists. Subsequent works showed that this number captures sensitive information about the ring, e.g., $s(R)$ (as a limit superior) is equal to 1 if and only if $R$ is regular \cite{HunekeLeuschke02}, is greater than zero if and only if $R$ is strongly $F$-regular \cite{AberbachLeuschke03}, and  is equal to the infimum of differences of Hilbert-Kunz multiplicities of a pair of nested $\idealm$-primary ideals \cite{WatanabeYoshida04, Yao06, PolstraTucker16}.

The $F$-signature of affine semigroup rings was studied in \cite{Singh05, WatanabeYoshida04, VonKorff11}. It was shown in these cases that the limit exists, is rational, and is equal to the volume of a certain rational polytope, the \emph{Watanabe-Yoshida-Von Korff polytope} constructed from $R$.

The existence of the limit in the definition of $s(R)$ was established by Tucker in 2012 \cite{Tucker12}. His approach uses a characterization of Yao \cite{Yao06} of the term $a_q(R)$ as the length of cyclic modules $R/I_q$ (see Section~\ref{applications}).
%\begin{equation}\label{splittingnumbers} 
%\begin{aligned} a_e =& \ p^{e \alpha} \cdot \ell_R(R / J_e)  \\ &\text{where} \quad  J_e = \{ r \in R \ | \ \phi(r) \in \idealm^{[p^e]} \ \forall \phi \in \mathrm{Hom}_{R^{p^e}}(R, R^{p^e})\} \end{aligned}
%\end{equation}
Then, using uniform estimates on Hilbert-Kunz multiplicity, it is shown that the sequence $q^{-d} \cdot \ell_R(R / I_q)$ tends to a limit if and only if the sequence $q^{-d} \cdot e_{HK}(I_q, R)$ does, and that the second sequence is nonincreasing and bounded below by zero. A simplified treatment of this proof, utilizing a lemma of Dutta, can be found in \cite{Huneke13}. These techniques have inspired further results, e.g., concerning the upper semicontinuity of Hilbert-Kunz multiplicity \cite{Smirnov14}, the lower semicontinuity of $F$-signature \cite{Polstra15}, and relations between Hilbert-Kunz multiplicity and $F$-signature \cite{PolstraTucker16}.

Another related limit, associated to an \emph{arbitrary} ideal $I$ of $R$ is the \emph{generalized Hilbert-Kunz multiplicity}, introduced by  Epstein and Yao~\cite{EpsteinYao}:
\begin{equation}\label{intro-gHK}
e_{gHK}(I,R) := \lim_{q \rightarrow \infty} \frac{\ell_R(H^0_\idealm(R/I^{[q]}))}{q} \, , % = \lim_{e \rightarrow \infty} \frac{  \ell_R(  (I^{[p^e]})^{\text{sat}}  /I^{[p^e]}  )  }{p^{ed}}  \,,
\end{equation}
provided the limit exists. This multiplicity is studied further by Dao and Smirnov \cite{DaoSmirnov14}, who show the limit above exists for all ideals in a complete intersection with an isolated singularity. They also show that, in such rings, vanishing of $e_{gHK}(I,R)$ characterizes ideals $I$ of finite projective dimension. This generalized multiplicity and related analogues are utilized by Brenner \cite{Brenner14} in his proof of the existence of irrational Hilbert-Kunz multiplicities. Vraciu \cite{Vraciu16} proved that, under the assumption that \emph{every} ideal of $R$ satisfies Huneke's condition (LC) (see Section~\ref{applications}), the generalized Hilbert-Kunz multiplicity exists. As this hypothesis on the ring is poorly understood, it is remains an open question whether this limit exists in general, and it is of interest to find other conditions under  which this limit exists.

\subsection{Numerical limits via Okounkov bodies}\label{limitsokounkovbodies} Another recent development concerning limits in algebra and algebraic geometry is the introduction of the method of Okounkov bodies. In \cite{Okounkov96} and \cite{Okounkov03}, Okounkov established a method for associating a convex body, now called an \emph{Okounkov body}, to the complete linear series $H^0(X,\mathcal{O}_X(mD))$ of an ample divisor $D$ on a smooth projective variety $X$. This geometric structure is then employed to establish many properties regarding the growth of these linear series, including the existence of limits and log-concavity. This construction of Okounkov bodies is generalized to incomplete linear series on big divisors with further applications in \cite{LazarsfeldMustata09}.

Similar constructions have been employed to study the growth of colengths of \emph{graded families of ideals}, i.e., sequences of ideals 
\[I_{\bullet} = \{ I_n \}_{n=1}^{\infty} \quad \text{with} \quad I_a \cdot I_b \subseteq I_{a+b} \ \text{for all} \ a, b \in \NN. \] 
In work of Ein, Lazarsfeld, and Smith \cite{EinLazarsfeldSmith03} and Musta\c{t}\u{a} \cite{Mustata02}, the limit
\begin{equation}\label{gradedfamily}
 \lim_{n\rightarrow\infty} \frac{\ell_R(R/I_n)}{n^{d}} \end{equation}
is studied and is shown to exist for every graded family $I_{\bullet}$ of $\idealm$-primary ideals in a regular local ring $(R, \idealm)$  containing a field. In \cite{KavehKhovanskii12}, Kaveh and Khovanskii apply valuation theory to construct convex bodies associated to graded families of ideals and graded algebras, as well as giving an alternative construction of Okounkov bodies for linear systems.

The approach of using valuation theory to constuct Okounkov bodies to study numerical limits measuring the growth of graded families in local rings was greatly extended in work of Cutkosky \cite{Cutkosky13, Cutkosky14}, who showed that, for any local ring $(R, \idealm)$, if the $R$-module dimension of the nilradical of $\hat{R}$ has dimension strictly less than that of $R$, then the limit (\ref{gradedfamily}) exists for \emph{every} graded family $I_\bullet$ of $\m$-primary ideals. Moreover, this is the strongest possible result, in the sense that if this limit exists for every such graded family, then the inequality on dimensions holds. A key application of this work is the existence of the $\varepsilon$-multiplicity of Katz and Validashti \cite{KatzValidashti10} and Ulrich and Validashti \cite{UlrichValidashti11} for all ideals in an analytically unramified ring.

Roughly speaking, the general approach of Cutkosky's existence theorem \cite{Cutkosky13, Cutkosky14} may be outlined as follows: First reduce to the case of a complete local domain. Then:\begin{enumerate}
\item\label{step1intro} Construct a suitable valuation on the ring $R$.
\item Observe that the values associated to elements in a graded family admit a combinatorial structure (semigroups).  
\item Compute the relevant $R$-module lengths by counting certain points in this combinatorial structure.
\item Show that the rate of growth of the number of elements in the combinatorial structure is equal to the Euclidean volume of a region (Okounkov bodies).
\end{enumerate}
The  steps above also form a rough outline for the existence and subadditivity results in \cite{Okounkov96, Okounkov03, LazarsfeldMustata09, KavehKhovanskii12}. One major difference between Okounkov bodies in the geometric setting of linear series on smooth varieties and in the local setting of graded families on possibly singular domains lies the first step. % Step~(\ref{step1intro}). 
  In particular, the standard construction of such a valuation  from a flag of smooth subvarieties as in \cite{LazarsfeldMustata09} only makes sense in the case $R$ is regular.

\subsection{Results in the present work}

The motivating idea behind this paper is to study numerical limits in positive characteristic by the technique of expressing such limits as Euclidean volumes as in Cutkosky's proof; that is, to apply the methods outlined in Subsection~\ref{limitsokounkovbodies} to the setting of Subsection~\ref{limitscharp}. One immediate obstacle is that the limits considered in Equations~ \eqref{HK},~\eqref{signature}, and~\eqref{intro-gHK} do not come as growth of colengths of graded families of ideals. However, the chains of ideals that do appear admit another type of structure, identified in the work of Tucker~\cite{Tucker12}:

\begin{definition} A \emph{$p$-family of ideals} is a sequence of ideals
\[I_{\bullet} = \{ I_{p^e} \}_{e=1}^{\infty} \quad \text{with} \quad I_q^{[p]}  \subseteq I_{pq} \ \text{for all} \ q \text{ a power of $p$}.\] 
\end{definition}

%\!{Examples of $p$-families of ideals include the iterated Frobenius powers of some fixed ideal, and the term-wise saturation of a $p$-family is also a $p$-family; we refer the reader to Section \ref{} for more examples.}

%\!{Define the volume of a $p$-family of $\m$-primary ideals}

Examples of $p$-families of ideals include the sequence of Frobenius powers of an ideal, the sequence obtained by taking the saturation of the Frobenius powers of an ideal, and the sequence of ideals appearing in the characterization of $F$-signature \cite[Lemma~4.4]{Tucker12}. 

Since we consider similar limits throughout, we will use a shorthand notation: If $(R,\m)$ is a local ring of characteristic $p>0$, and $I_{\bullet}$ is a $p$-family of $\m$-primary ideals of $R$, then the \emph{volume} of $I_{\bullet}$  is
\[\vol_R(I_{\bullet}) := \lim_{q\rightarrow\infty} \frac{\ell_R(R/I_q)}{q^d} \]
where $q$ varies through powers $p^e$ of the characteristic $p$.

Given this terminology, we may state our main results.

\begin{theorem}\label{existence-intro} Let $(R,\idealm)$ be a ring of prime characteristic $p>0$ and Krull dimension $d$. The limit $\vol_R(I_{\bullet})$ exists for every $p$-family $I_{\bullet}$ of $\m$-primary ideals of $R$ if and only if the dimension (as an $R$-module) of the nilradical of the completion of $R$ is less than $d$. 

Moreover, if $R$ is reduced and excellent, and $I_{\bullet}$ and $J_{\bullet}$ are $p$-families of ideals of $R$ such that $ I_q \subseteq J_q$ and  ${I_q \cap \idealm^{cq} = J_q \cap \idealm^{cq}}$ for every $q$ and some $c$ independent of $q$, then the limit \[ \lim_{q\rightarrow \infty} \frac{\ell_R(J_q/I_q)}{q^d}\] exists.
\end{theorem}

The following example illustrates the necessity of the hypothesis on the dimension of the nilradical of the completion of $R$.

\begin{example} 
\label{volume DNE: E}
The nilradical of $R = \FF_p \llbracket x,y \rrbracket / y^2$ is generated by the class of $y$, and so its $R$-module dimension is $1$.  Let $\{ k_q \}_{q=1}^{\infty}$ be a sequence of integers indexed by powers $q=p^e$ of $p$ with $0 \leq k_q < q$, and for every $q$ let  $I_q$ be either the ideal generated by the class of $x^q$, or by the classes of $x^q$ and $x^{k_q} y$.  Though we allow this choice to vary with $q$, the Frobenius power $I_q^{[p]}$ is always generated by the class of $x^{pq}$, and so the sequence $\{I_q\}_{q=1}^{\infty}$ is a $p$-family in $R$.   However, depending on the choice of $I_q$, the length of $R/I_q$ equals either $2q$ or $q + k_q$.  Thus, there are many $p$-families in $R$ for which the limit $\vol_R(I_{\bullet})$ does not exist.  
\end{example}

The proof of the positive part of this result follows a different method than that of \cite{Tucker12}, yielding the new existence result for pairs of $p$-families. Key corollaries of this theorem include new cases of the existence of generalized Hilbert-Kunz multiplicity, e.g., for ideals of dimension one in graded rings, and a new uniform proof of the existence of both Hilbert-Kunz multiplicity and $F$-signature. Moreover, our proof of the above theorem realizes the limit as the Euclidean volume of a region in $\RR^d$ that in the case of $F$-signature of normal semigroup rings agrees with the Watanabe-Yoshida-Von Korff polytope.

The idea behind the proof is in fact very simple and illuminating. As with Cutkosky's existence proofs, we utilize a suitable valuation, identify a natural combinatorial structure on the set of values that arise (a \emph{p-system} of semigroup ideals), and prove that the growth rate of the number of elements in the combinatorial structure is equal to the volume of a region in Euclidean space (a \emph{p-body}). The main technical mass of the paper consists in this last step, other than reductions to the complete domain case that follow along similar lines as Cutkosky's arguments.

While the regions we obtain are not convex (in contrast with Okounkov bodies arising from graded families, which are convex), they are convex relative to a cone. This structure allows us to find new relations on $p$-families and limits in positive characteristic. For example, we obtain an analogue of the Brunn-Minkowski inequality:

\begin{theorem}\label{BM-ineq} 
Let $I_\bullet$ and $J_\bullet$ be $\idealm$-primary $p$-families in a $d$-dimensional excellent local domain $(R,\idealm)$ of characteristic $p>0$. If $L_{\bullet}$ is the $p$-family whose $q$-th term is the product of  $I_q$ and $J_q$, then  
\[ \vol_R(I_{\bullet})^{1/d}+ \vol_R(J_{\bullet})^{1/d} \geq \vol_R(L_{\bullet})^{1/d}.\]
\end{theorem}

We also obtain a uniformity result on $p$-families of ideals:

\begin{theorem}\label{nonvanishing-intro} Let $I_\bullet$ be an $\idealm$-primary $p$-family of ideals in a $d$-dimensional excellent local  domain $(R,\idealm)$ of characteristic $p>0$. Then  the limit $\vol_R(I_\bullet)$ is positive if and only if there exists a power $q_{\circ}$ of $p$ such that $I_{qq_{\circ}} \subseteq \idealm^{[q]}$ for all $q=p^e$.
\end{theorem}

In Section~\ref{Background}, we review some basic notions in convex Euclidean geometry, semigroups, and valuations that we will use throughout. We also recall here a couple of essential theorems of Kaveh and Khovanskii. In Section~\ref{OK-valuations}, we define a class of valuations that are suitable for measuring colengths of ideals. Many of the ideas in this section come from Kaveh-Khovanskii and Cutkosky; however, we include details here to keep the presentation self-contained. In Section~\ref{semigroups-p-stuff}, we describe the key combinatorial structure, a $p$-system, and the region in Euclidean space whose volume measures it, a $p$-body. In Section~\ref{limits-as-volumes}, we prove Theorem~\ref{existence-intro} on the existence of limits. It appears as the combination of Theorem~\ref{OK limit existence: T}, which relates asymptotic growth of lengths to volumes of $p$-bodies, Theorem~\ref{m-primary limits exist: T}, establishing the ``if and only if'' part of the statement, and Theorem~\ref{limits-of-pairs-reduced}, which corresponds to the last assertion. Theorem~\ref{BM-ineq} also appears there as Theorem~\ref{BMineq}, and Theorem~\ref{nonvanishing-intro} appears as Theorem~\ref{vanishing: T}. In Section~\ref{applications}, we apply these existence results to the particular limits mentioned in the introduction, and discuss our work in the toric context.

%%%%%%%%%%%%%%%% SECTION %%%%%%%%%%%%%%%%%%%%

\section{Conventions and basic notions}\label{Background}

In this section, we establish some notation, and review some basics from convex geometry, semigroup theory, and valuation theory.  Throughout this discussion, $U$ and $V$ are arbitrary subsets of $\RR^d$.

\subsection{Euclidean geometry and measure}  The standard inner product of vectors $\vv{u}$ and $\vv{v}$ in $\RR^d$ is denoted $\iprod{\vv{u}}{\vv{v}}$.   We use $\interior{\vs{U}}, \overline{\vs{U}}$, and $\boundary{\vs{U}} = \overline{\vs{U}} \setminus \interior{\vs{U}}$  to denote the interior, closure, and boundary, respectively, of $U$ with respect to the Euclidean topology on $\RR^d$.   If $U$ is Lebesgue measurable, we call its measure $\vol_{\RR^d}(U)$ its \emph{volume}.    Given a real number $\lambda$,  the \emph{Minkowski sum} $U + \lambda V$ consists of all sums $\vv{u}+\lambda \vv{v}$ with $\vv{u} \in U$ and $\vv{v} \in V$.  We define $U-\lambda V$ analogously.

\subsection{Convex cones}  

A \emph{conical combination} of points in $\RR^d$ is an $\RR$-linear combination of those points with nonnegative coefficients, and a  \emph{convex cone} (or simply, \emph{cone}) is any subset of $\RR^d$ that is closed under taking conical combinations.  The \emph{closed cone generated by $U$} consists of the closure of the set of all conical combinations of elements of $U$, and is denoted $\cone(U)$.  Recall that a cone in $\RR^d$ has a nonempty interior if and only if the real vector space it generates has dimension $d$.  We call such a cone \emph{full-dimensional}. 

 A cone $C$ is \emph{pointed} if it is closed, and if there exists a  vector $\vv{a} \in \RR^d$ such that $ \iprod{\vv{u}}{\vv{a}} >0$ for all nonzero points $\vv{u}$ in $C$.  In this case,  if $\alpha$ is a nonnegative real number, we call any halfspace of the form
\[ H = \{ \vv{u} \in \RR^d : \iprod{\vv{u}}{\vv{a}} < \alpha \} \] a \emph{truncating halfspace} for $C$.    Observe that if $\alpha$ is positive, then the  intersection of $C$ and $H$ is nonempty and bounded, and we call any subset of $C$ of this form a \emph{truncation} of $C$.

Given a cone $C$ in $\RR^d$, we say that a subset $\Gamma$ of $\RR^d$ is a \emph{$C$-convex region} (or simply \emph{$C$-convex}) if the Minkowski sum $C+\Gamma$ lies in $\Gamma$.
\begin{example}
\label{C-convex region: R}
  If $C$ is a pointed cone in $\RR^d$, and $H$ is a truncating halfspace for $C$, then the region $\Gamma$ consisting of all points in $C$ lying outside of the truncation $C \cap H$ is a $C$-convex region. 
  \end{example}

We recall the following theorem of Kaveh and Khovanskii \cite[Corollary~2.4]{KavehKhovanskii14}, which is reminiscent of the classical Brunn-Minkowski theorem for convex sets.

\begin{theorem}\label{brunn-minkowski-convex}
Let $C$ be a full-dimensional pointed cone in $\RR^d$, and let $\Delta'$ and $\Delta''$ be $C$-convex subsets of $C$.  If the complements of $\Delta'$ and $\Delta''$ in $C$ are bounded, then
\[ \vol_{\RR^d}( C \setminus \Delta' ) ^{1/d} + \vol_{\RR^d}( C \setminus \Delta'' ) ^{1/d} \geq \vol_{\RR^d}( C \setminus (\Delta' + \Delta'') ) ^{1/d}.\]
\qed
\end{theorem}

We shall also need the following lemma.

\begin{lemma}\label{volume-zero} Let $C$ be a full-dimensional pointed cone in $\RR^d$, and let $H$ be  a truncating halfspace for $C$.  Let $\Delta$ be a $C$-convex subset of $C$ such that $\vol_{\RR^d}(\partial \Delta)=0$. Then ${\vol_{\RR^d}( C \setminus \Delta ) >0}$ if and only if there is some $\alpha>0$ such that $\Delta \cap \alpha H$ is empty.

\end{lemma}
\begin{proof} Fix a vector $\vv{a}$ in $\RR^d$ such that $H=\{ \vv{u}\in \RR^d \ | \ \iprod{\vv{u}}{\vv{a}}<1\}$.  If there exists an $\alpha >0$ such that $\Delta \cap \alpha H$ is empty, then $C \setminus  \Delta$ contains $C \cap \alpha H$, which has positive volume.

On the other hand, if $\vol_{\RR^d}(C \setminus \Delta)$ is positive, then the assumption that $\vol_{\RR^d}(\partial \Delta)=0$ implies that $\vol_{\RR^d}( C \setminus \bar{\Delta} ) >0$.  Now, suppose for the sake of contradiction that there does not exist a number $\alpha>0$ such that $\Delta \cap \alpha H$ is empty. By our choice of $\vv{a}$, this implies that there is a sequence of points $\{\vv{v}_n \}_{n=1}^{\infty}$  in $\Delta$ such that $\iprod{\vv{v}_n}{\vv{a}}\rightarrow 0$. For all large $n$, the term $\vv{v}_n$ lies in the compact set $C\cap H$,  and therefore a  subsequence of $\{\vv{v}_n\}_{n=1}^{\infty}$ accumulates to a point $\vv{v}\in C$ with ${\iprod{\vv{v}}{\vv{a}}= 0}$. However, $C$ is pointed with $H$ a truncating halfspace, and so the only point in $C$ whose inner product with $\vv{a}$ is zero is the origin.  Thus, we must have that $\vv{v}=\vv{0}$, which implies that $\vv{0}\in \bar{\Delta}$. Finally, since $\Delta$ is $C$-convex, this implies that $C=\bar{\Delta}$, which contradicts that $\vol_{\RR^d}( C \setminus \bar{\Delta} ) >0$.
 \end{proof}

\subsection{Semigroups}     
In this article, a \emph{semigroup} is any subset of $\ZZ^d$ that contains zero and is closed under addition.  The semigroup of nonnegative integers is denoted $\NN$, and a semigroup $S$ is \emph{finitely generated} if there exists a finite subset $G$ of $S$ such that every element of $S$ can be written as an $\NN$-linear combination of elements of $G$.  We call a subset $T$ of a semigroup $S$ an \emph{ideal} of $S$ whenever $S+T$ is contained in $T$.  

A semigroup $S$ is called \emph{pointed} if the only element of $S$ whose negative is also in $S$ is zero.  Note that if the closed cone generated by $S$ is pointed, then so is $S$.  However, the converse is false, but does hold if $S$ is assumed to be finitely generated.

Observe that the group generated by a semigroup $S$ in $\ZZ^d$ is the Minkowski difference $S-S$.  Thus, after replacing $\ZZ^d$ with this finitely generated free abelian group, when convenient, we may always assume that $S$ lies in $\ZZ^d$, and that $S-S = \ZZ^d$.  

We conclude our discussion of semigroups by recalling an important relation between semigroups and cones, which is a special case of \cite[Theorem 1.6]{KavehKhovanskii12}, also due to Kaveh and Khovanskii.

\begin{theorem}[Kaveh and Khovanskii]
\label{ConeTheorem: T}
Let $S$ be a semigroup in $\ZZ^d$ such that $S-S = \ZZ^d$, and such that the full-dimensional cone $C$ generated by $S$ is pointed.  If $\mathscr{C}$ is any full-dimensional pointed subcone of $C$ whose boundary intersects $\partial C$ only at the origin, then there exists truncating halfspace $\mathscr{H}$ for $C$ such that $(\mathscr{C} \cap \ZZ^d) \setminus \mathscr{H}$ lies in $S$. \qed 
\end{theorem}

\subsection{Valuations}  Herein, we consider valuations whose value group is order-isomorphic to a  finitely generated free additive subgroup of $\RR$.  In particular, all valuations will be of rank one and of finite rational rank.   Given such a valuation, we will be interested in certain semigroup-theoretic  and convex-geometric  objects defined in terms of the value group, and will therefore find it convenient to fix an isomorphism  of this group with the lattice points in some Euclidean space. 

Towards this, fix a $\ZZ$-linear embedding of $\ZZ^d$ into $\RR$, and consider the inherited linear order on $\ZZ^d$.  A  \emph{valuation} on a field $\FF$ with value group $\ZZ^d$ is a surjective group homomorphism \[ \nu: \FF^{\times} \longonto \ZZ^d \] 
with the property that $\nu(x+y)$ is greater than or equal to the minimum of $\nu(x)$ and $\nu(y)$ for all $x$ and $y$ in $\FF^{\times}$ not differing by a sign.  In an abuse of notation, given a subset $M$ of $\FF$, we will set \[ \nu(M) := \nu(M \setminus 0).\] 

Given a point $\vv{u}$ in $\ZZ^d$,   we use $\FF_{\geq \vv{u}}$ to denote the subset of $\FF$ consisting of $0$ and the elements $x$ with $\nu(x) \geq \vv{u}$, and  we define $\FF_{ > \vv{u}}$  similarly.   

The \emph{local ring of $\nu$} is the subring $(V, \m_V, \kk_V)$ of $\FF$ consisting of the valuation ring $\FF_{\geq \vv{0}}$ and its unique maximal ideal $\FF_{> \vv{0}}$.  

Recall that $\nu$ is said to \emph{dominate} a local domain $(D, \m, \kk)$ with fraction field $\FF$ if $(D, \m, \kk)$ is a local subring of $(V, \m_V, \kk_V)$.

\begin{remark} 
\label{dimension of quotient: R}
 If $\vv{u} \in \ZZ^d$, then both $\FF_{ > {\vv{u}}}$ and $\FF_{\geq \vv{u}}$ are naturally  modules over $V$, and $\FF_{\geq \vv{u}} \, / \, \FF_{ > \vv{u}}$ is a vector space over the residue field $\kk_V$.   

In fact,  if $\vv{u} = \nu(h)$ for some nonzero $h$, then \[ \dim_{\kk_V}  \( \frac{\FF_{\geq \vv{u}}}{\FF_{ > \vv{u}}} \) = 1.\]  
Indeed, this follows from the observation that if $g$ is any other nonzero $g$ with $\vv{u} = \nu(g)$, then $g$ differs from $h$ by $g/h$, which has $\nu$-value $\vv{0}$.
\end{remark}

%\begin{remark}[On terminology]
%\label{notion of valuation} In the standard language of valuation theory,  our notion of a valuation on a domain $D$ with values in the $\ZZ$-linearly embedded subgroup $\ZZ^d \longhookrightarrow \RR$ corresponds to a valuation on the fraction field of $D$ of rank one and rational rank at most $d$.   %Such a valuation $\mu$ by definition takes values in a free abelian subgroup of $\RR$ that has rank at most $d$. Setting $\vv{a}$ to be a vector whose entries are positive and form a basis for this value group, one obtains a valuation with values in $\ZZ^d$ by taking $\nu(x)$ to be the $\vv{a}$-coordinates of $\mu(x)$.
%\!{I removed the domination condition, since I don't think we are actually requiring it. }
%\end{remark}

The linear order on $\ZZ^d$ resulting from its embedding into $\RR$ is a crucial ingredient in our definitions above.  For future reference, we record a concrete description of this order below.

\begin{remark}
\label{embedding into RR: R}
  Every $\ZZ$-linear embedding $\ZZ^d \longhookrightarrow \RR$ is of the form $\vv{u}~\mapsto~\iprod{\vv{a}}{\vv{u}}$ for some unique vector $\vv{a} \in \RR^d$ whose coordinates are linearly independent over $\QQ$.  Consequently, the inherited order on $\ZZ^d$ is such that $\vv{u} \leq \vv{v}$ in $\ZZ^d$ whenever $\iprod{\vv{u}}{\vv{a}} \leq \iprod{\vv{v}}{\vv{a}}$ in $\RR$.
\end{remark}

\subsection{Interplay}  Let $D$ be a domain of dimension $d$ with fraction field $\FF$.  Fix a $\ZZ$-linear embedding $\ZZ^d \longhookrightarrow \RR$ induced by a vector $\vv{a} \in \RR^d$, and consider the linear order on $\ZZ^d$ as described in Remark \ref{embedding into RR: R}.  Let 
\[ \nu: \FF^{\times} \longonto \ZZ^d \] be a valuation on $\FF$ with value group $\ZZ^d$.

As $\nu$ is a group homomorphism, $S= \nu(D)$ is a semigroup in $\ZZ^d$, and the image under $\nu$ of any  ideal of $D$ is a semigroup ideal of $S$.   Furthermore, the surjectivity of $\nu$ implies that $S-S = \ZZ^d$, and so $S$ generates the ambient $\ZZ^d$ as a group.  It follows that the closed cone $C$ in $\RR^d$ generated by $S$ is full-dimensional.  We stress that, in practice, $S$ is often not finitely generated, and so $C$ need not be polyhedral.

Suppose that $(D, \m, \kk)$ is local.  In terms of the semigroup $S=\nu(D)$, it is not difficult to see that $D$ is dominated by $\nu$ if and only if every element of $\nu(\m)$ is greater than $\vv{0}$, and $S = \nu(\m) \cup \{ \vv{0} \}$.  In light of Remark \ref{embedding into RR: R}, this implies that every nonzero point $\vv{u}$ of $S$ satisfies $\iprod{\vv{u}}{\vv{a}}  > 0$, and so $S$ is pointed by the order on $\ZZ^d$.  

These observations motivate the following (nonstandard) definition.

\begin{definition}[Strong domination] In the above notation, we say that $D$ is \emph{strongly dominated} by $\nu$ if it is dominated by $\nu$, and the full-dimensional closed cone $C$ generated by $S$ (and not just $S$ itself) is pointed by the order on $\ZZ^d$.  In other words, $D$ is strongly dominated by $\nu$ if it is dominated by $\nu$, and $\iprod{\vv{u}}{\vv{a}} > 0$ for every nonzero $\vv{u}$ in $C$.

\end{definition}

%%%%%%%%%%%%%%%%%%%%%% SECTION %%%%%%%%%%%%%%%%%%

\section{OK valuations}
\label{OK-valuations}

Following Cutkosky~\cite{Cutkosky13, Cutkosky14}, we consider a distinguished class of valuations that extend the notion of ``good valuation'' with ``$1$-dimensional leaves" defined by Kaveh and Khovanskii in \cite{KavehKhovanskii12}.  For the convenience of the reader, we review the relevant constructions in this section.

\subsection{OK valuations} 

\begin{definition}
\label{OKvaluation: D}    Let $(D, \m, \kk)$ be a local domain of dimension $d$ with fraction field $\FF$, and fix a $\ZZ$-linear embedding of $\ZZ^d$ into $\RR$.   

A valuation \[ \nu: \FF^{\times} \longonto \ZZ^d \] on $\FF$ with value group $\ZZ^d$ and local ring $(V, \m_V, \kk_V)$ is said to be \emph{OK relative to D} whenever $D$ is strongly dominated by $\nu$,  the resulting extension of the residue fields $\kk \longhookrightarrow \kk_V$ is finite (algebraic), and 
there exists a point $\vv{v}$ in $\ZZ^d$ such that \[ D  \cap \  \FF_{\geq a \vv{v}} \subseteq \m^a \]  for every nonnegative integer $a$. 
\end{definition}

\begin{definition} A  local domain $D$ of dimension $d$ is said to be \emph{OK} whenever there exists an valuation on its fraction field with value group $\ZZ^d$ that is OK relative to $D$.
\end{definition}

%\begin{remark}
%As in Remark \ref{embedding into RR: R},  let $\vv{a}$ in $\RR^d$ be the unique vector determining the fixed $\ZZ$-linear embedding of $\ZZ^d$ into $\RR$.   With this notation, the last condition in Definition \ref{OKvaluation: D} can be recast as follows:  If $C$ is the full-dimensional cone generated by the semigroup $S = \nu(D)$, then there exists a point $\vv{v} \in S$ such that if $H$ is the halfspace of all $\vv{u} \in \RR^d$ with $\iprod{\vv{u}}{\vv{a}} < \iprod{\vv{v}}{\vv{a}}$, and $x \in D$ is such that $\nu(x)$ does not lie in the truncated cone $C \cap aH$, then $x \in \m^a$. 
%\end{remark}

\begin{example} 
\label{toric ring is OK: E} 
Let $S$ be a finitely generated pointed semigroup in $\ZZ^d$ with $S-S = \ZZ^d$.  As $S$ is finitely generated, the full-dimensional cone $C$ generated by $S$ is also pointed, and so  we may fix a point $\vv{a} \in \RR^d$ such that $\iprod{\vv{u}}{\vv{a}}$ is positive for every nonzero point $\vv{u}$ in $C$.  As $C$ is full-dimensional,  we may perturb the coordinates of $\vv{a}$ so that they are linearly independent over $\QQ$.  In what follows, we consider the $\ZZ$-linear embedding of $\ZZ^d$ into $\RR$ determined by this point.    

Let $\kk[S]$ be the $\kk$-subalgebra of $\kk[x_1^{\pm 1}, \dots, x_d^{\pm 1} ]$ with $\kk$-basis given by the monomials $x^{\vv{u}}$ with $\vv{u} \in S$, $D$ be the localization of $\kk[S]$ at the maximal ideal generated by the monomials with nonzero exponents.

Given a nonzero element $f$ in $\kk[S]$, let $\nu(f)$ be the least exponent that corresponds to a supporting monomial of $f$.  We claim that the induced valuation $\nu: \FF^{\times} \longonto \ZZ^d$ on the fraction field $\FF$ of $D$ is OK relative to $D$.  Indeed, this essentially follows from the observation that the semigroup $\nu(D)$ agrees with the semigroup $S$, and from the definition of the order on $\ZZ^d$.  We note that the extension of residue fields resulting from the domination of $D$ by $\nu$ is an isomorphism, and that one may take $\vv{v}$ in Definition \ref{OKvaluation: D} to be the greatest element of some finite generating set for $S$.
\end{example}

\begin{example}[A power series ring over a field is OK]
\label{power series ring is OK: E}
Let $D$ be the power series ring in $d$ variables over a field $\kk$, and with fraction field $\FF$.  Fix a vector $\vv{a}$ in $\RR^d$ with \emph{positive} entries that are linearly independent over $\QQ$, and consider the associated embedding $\ZZ^d \longhookrightarrow \RR$. 

The positivity of $\vv{a}$ implies that every subset of $\NN^d$ has a least element, and so the function $\nu$ assigning to each nonzero power series in $D$ the least exponent vector that appears among its terms is well-defined.  We claim that the induced valuation $\nu: \FF^{\times} \longonto \ZZ^d$ is OK relative to $D$.  For example, the fact that $C$ is contained in the nonnegative orthant tells us that $\iprod{\vv{u}}{\vv{a}}$ is positive for all nonzero points $\vv{u}$ in $C$.  The remaining details, similar to those in Example \ref{toric ring is OK: E}, are left to the reader.
\end{example}

\begin{example}[A regular local ring containing a field is OK]
\label{RLR is OK: E}
If $(D, \m, \kk)$ is a regular local ring containing a field,  then we may identify its completion $\hat{D}$ with a power series over $\kk$ in some minimal generating set for $\m$, and it is straightforward to verify that the restriction to the fraction field of $D$ of the valuation on the fraction field of $\hat{D}$ constructed in Example \ref{power series ring is OK: E} is OK relative to $D$.
\end{example}

Example \ref{RLR is OK: E} suggests that if an extension of rings is nice, then the restriction of an OK valuation is often OK.  Below, we see another important instance of this theme.

\begin{lemma}  
\label{restriction of OK is OK: L}
Consider an inclusion of local domains $(D, \m) \subseteq(A, \n)$ such that $\hat{D}$ is reduced (e.g., $D$ is an excellent domain), $A$ is essentially of finite type over $D$, and the induced extension of fraction fields is an isomorphism.  If $\FF$ is the common fraction field of $A$ and $D$,  then the restriction of any valuation on $\FF$ that is OK relative to $A$ is also OK relative to $D$.
\end{lemma}

\begin{proof} Fix a valuation $\nu: \FF^{\times} \longonto \ZZ^d$ that is OK relative to $A$.  To see that $\nu$ is also OK relative to $D$, first observe that the cone generated by $\nu(D)$ is contained in the one generated by $\nu(A)$.  Moreover, $D$ is a local subring of $A$, and hence, a local subring of the valuation ring $V$, and since $A$ is essentially of finite type over $D$ and $\nu$ is OK relative to $A$, the composition of extensions $D/\m \hookrightarrow A/\n \hookrightarrow \kk_{\nu}$ is finite.   It remains to verify the last condition concerning the existence of a distinguished vector in $\ZZ^d$, which we do below.

 If $D$ is complete, or more generally, excellent, then \cite[Theorem~1]{Hubl01}
shows that there exists a positive integer $l$ such that 
\[ \n^{kl} \cap D \subseteq \m^{k}   \] for every $k \in \NN$. An  argument appearing the proof of \cite[Lemma~4.3]{Cutkosky14}
shows that the same is also true if instead we only suppose that $\hat{D}$ is reduced.  In either case, if the point $\vv{v}$ satisfies the last condition in Definition \ref{OKvaluation: D} relative to $A$, then the point $\vv{w} =l\vv{v}$ satisfies this same condition relative to $D$.
\end{proof}

\begin{corollary}
\label{complete local domain is OK: C}
An excellent local domain containing a field is OK. In particular, a complete local domain containing a field is OK. 
\end{corollary}

\begin{proof}  The following construction appears in the proof of \cite[Corollary 4.3]{Cutkosky14}.  
Fix an excellent local domain $(D, \m)$, and let $\pi:  X  \to \operatorname{Spec} D$ be the normalization of the blowup of $\operatorname{Spec} D$ at $\m$.  As $X$ is normal (and hence, regular in codimension one), there exists a point $x \in X$ lying over $\m$ such that $A = \mathcal{O}_{X, x}$ is regular.    Moreover, since $D$ is excellent, we know that $A$ is essentially of finite type over $D$.  Example \ref{RLR is OK: E} tells us that $A$ is OK, and Lemma \ref{restriction of OK is OK: L} then implies that $D$ is OK.
\end{proof}

\subsection{Length computations}

In this subsection, we fix a valuation $\nu$ on the fraction field $\FF$ of a $d$-dimensional local domain $(D, \m, \kk)$ that is OK relative to $D$, and we adopt the notation established in Definition \ref{OKvaluation: D}.  Moreover, we use $S$ to denote the semigroup $\nu(D)$ in $\ZZ^d$.

\begin{remark}
\label{size of preimage: R}
  By assumption, $D$ is a local subring of $V$.  Therefore,  if $M$ is an $D$-submodule of $\FF$ and $\vv{u} \in \ZZ^d$, then the $D$-module structure on $M$ induces a $\kk$-vector space structure on the quotient
\[ \frac{ M \cap \FF_{ \geq \vv{u}}}{ M \cap \FF_{> \vv{u}}}.  \] 

Moreover, the obvious $\kk$-linear embedding 
\[ \frac{ M \cap \FF_{ \geq \vv{u}}}{ M \cap \FF_{> \vv{u}}}  \longhookrightarrow  \frac{ \FF_{\geq \vv{u}} }{\FF_{>\vv{u}}}\] and Remark \ref{dimension of quotient: R} show that the $\kk$-dimension of this quotient is at most most $[\kk_{V} : \kk]$.    By definition, this space is nonzero if and only if there exists an element $m \in M$ with $\nu(m) = \vv{u}$, and motivated by this, we may regard its $\kk$-dimension as a ``measure" of the size of the set of all elements in $M$ with $\nu$-value equal to $\vv{u}$. 
\end{remark}

\newcommand{\Th}{\nu^{ \hspace{.01in} (h)}}
\newcommand{\Tl}{\nu^{\hspace{.01in} (l)}}
\newcommand{\Tone}{\nu^{(1)}}
\renewcommand{\T}{\nu}

\begin{definition}
\label{T ideals: D}
If $M$ is a $D$-submodule of $\FF$ and $1 \leq h \leq [ \kk_{V} : \kk]$, then we define $\Th(M)$ to be the set of all vectors $\vv{u} \in \ZZ^d$ such that  
\[ \dim_{\kk} \( \frac{ M \cap \FF_{ \geq \vv{u}}}{ M \cap \FF_{> \vv{u}}} \) \geq h .\]  
As noted in Remark \ref{size of preimage: R}, we have that $\Tone(M) = \nu( M)$. 
\end{definition}

\begin{remark}
\label{T defines ideals: R}
Suppose $M$ is a $D$-submodule of $\FF$.  If $\vv{u} \in \ZZ^d$ and $g$ is a nonzero element of $D$ with $\vv{v}=\nu(g)$, then the map 
\[  \frac{ M \cap \FF_{ \geq \vv{u}}}{ M \cap \FF_{> \vv{u}}}  \longrightarrow   \frac{ M \cap \FF_{ \geq \vv{u} + \vv{v} }}{ M \cap \FF_{> \vv{u}+\vv{v}}} \] 
defined by $[m] \mapsto [gm]$ is a $\kk$-linear injection.  In particular, 
 \[ \Th(M) + S \subseteq \Th(M), \] and so $\Th(M)$ is an ideal of $S$ for every $1 \leq h \leq [\kk_{V}: \kk]$.  
 \end{remark}

\begin{lemma}  
\label{lengths via counting points: L}
If $M$ is a $D$-submodule of $\FF$ and $\vv{v} \in \ZZ^d$, then \[ \ell_D \left( M / M \cap \FF_{\geq \vv{v}} \right) = \sum \limits_{h=1}^{[ \kk_{V}: \kk]}  \# ( \Th(M) \cap \mathscr{H} ),\] where $\mathscr{H}$ is the halfspace $\{ \vv{u} \in \RR^d :  \iprod{\vv{u}}{\vv{a}} < \iprod{\vv{v}}{\vv{a}} \} $. 
\end{lemma}

\begin{proof}   Set $l= [\kk_{V} : \kk]$.  By definition of the descending chain of semigroup ideals $\T (M) \supseteq \cdots \supseteq \Tl(M)$ of $S$,  if $\vv{u} \in \T(M)$, then 
\[  \dim_{\kk} \(  \frac{ M \cap \FF_{ \geq \vv{u}}}{ M \cap \FF_{> \vv{u}}} \) =  \# \set{ 1 \leq h \leq l : \vv{u} \in \Th(M)}. \]

As $\mathscr{H}$ is a truncating halfspace for $C$, the intersection $\T(M) \cap \mathscr{H}$ is finite.  If this set is empty, then the lemma is trivial.  Otherwise, this set consists of points  $\vv{u}_m > \cdots > \vv{u}_1$.  By definition of the order on $\ZZ^d$, the point $\vv{u}_m$ is the maximal vector in $\T (M)$ less than $\vv{v}$, and so $M \cap \FF_{\geq \vv{v}} = M \cap \FF_{> \vv{u}_m}$.  Consequently, the length of $M / (M \cap \FF_{\geq \vv{v}})$ equals the length of $M / (M \cap \FF_{> \vv{u}_m})$, which in turn equals 
\[ \ell_D \( \frac{ M \cap \FF_{\geq \vv{u}_m}}{M \cap \FF_{> \vv{u}_m}} \)  + \ell_D \( \frac{ M}{M \cap \FF_{\geq \vv{u}_m}} \). \]

Inducing on $m$ then shows that 
\[  \ell_D \( \frac{ M}{M \cap \FF_{\geq \vv{v}}} \) = \sum_{k=1}^m  \ell_D \( \frac{ M \cap \FF_{\geq \vv{u}_k}}{M \cap \FF_{> \vv{u}_k}} \),  \] 
which by an earlier observation equals
\[  \sum_{k=1}^m \#  \set{ 1 \leq h \leq l : \vv{u}_k \in \Th(I)}  .\] 

Finally, setting $E_{k,h} = 1$ if $\vv{u}_k \in \Th(I)$ and $E_{k,h} = 0$ otherwise, we may rewrite this expression to see that
$\ell_D \( { M} / {M \cap \FF_{\geq \vv{v}}} \)$ equals
\[ \sum_{k=1}^m \sum_{h=1}^l E_{k,h}  = \sum_{h=1}^l \sum_{k=1}^m E_{k,h} = \sum_{h=1}^{l}  \size{\( \Th(I) \cap \mathscr{H} \)}.
\] 
\end{proof}

In the next lemma, we see that the ideals $\Th(M)$ can be approximated by the ideals $\T(M)$ in a uniform way.  

\begin{lemma} 
\label{approximation: L}
There exists an element $\vv{u}$ in $S$ such that \[ \T(M) + \vv{u} \subseteq \Th(M) \subseteq \T(M) \] for every $D$-submodule $M$ of $\FF$ and integer $1 \leq h \leq [\kk_{V} : \kk]$.
\end{lemma}

\begin{proof}  Set  $l = [\kk_{V} : \kk]$.  By definition,  $\Th(M)$ lies in $\T(M)$.  To establish the opposite inclusion,  fix elements $g_1, \dots, g_l$ and $g$ in $D$ with \[ \vv{u} := \nu(g_1) = \cdots = \nu(g_l) = \nu(g), \] and such that the classes of the fractions $g_1/g, \dots, g_l/g$ form a basis for $\kk_{V}$ over $\kk$.   If $m$ is a nonzero element in $M$ with $\vv{v} = \nu(m)$, then it is not difficult to see that the classes of $g_1 m, \dots, g_l m$ in  \[  \frac{ M \cap \FF_{ \geq \vv{u} + \vv{v}}}{ M \cap \FF_{> \vv{u} + \vv{v}}}\]
are linearly independent over $\kk$, and hence, form a $\kk$-basis for this space.  As $m \in M$ was arbitrary, this shows that $\vv{u} + \T(M)$ is contained in $\Tl(M)$, and therefore, lies in $\Th(M)$ for all $1 \leq h \leq l$.
\end{proof}

%%%%%%%%%%%%%%%%%% SECTION %%%%%%%%%%%%%%%%%%%%

\section{Semigroups, $p$-systems, and $p$-bodies}\label{semigroups-p-stuff}

\begin{definition}
A semigroup $S$ in $\ZZ^d$ is \emph{standard} if $S-S = \ZZ^d$, and the full-dimensional cone generated by $S$ in $\RR^d$ is pointed.
\end{definition}

\begin{example}
In this paper, the main example of a standard semigroup in $\ZZ^d$ is the semigroup $\nu(D)$ associated to a valuation $\nu$ that is OK relative to some $d$-dimensional local domain $D$. 
\end{example}

\begin{definition}
\label{pSystem: D}
A sequence of ideals $T_{\bullet} = \set{ T_{q}}_{q=1}^{\infty}$ of a semigroup whose terms are indexed by the powers of $p$ is called a \emph{$p$-system} whenever $pT_{q}$ lies in $T_{pq}$ for every $q$ a power of $p$.
\end{definition}

\begin{definition}
\label{pBody: D}
Let $S$ be a standard semigroup in $\ZZ^d$.  Given a $p$-system of ideals $T_{\bullet}$ of $S$, we set 
\[ {\Delta}_{q}(S,T_{\bullet})  = \frac{1}{q} \cdot T_{q} + \cone(S), \] 
and we call the ascending union of sets
\[{\Delta}(S,T_{\bullet}) = \bigcup_{q=1}^{\infty} \Delta_{q}(S,T_{\bullet})\]
the \emph{$p$-body} associated to $(S, T_{\bullet})$.   Equivalently,  \[ \Delta(S, T_{\bullet}) = \left( \ \bigcup_{q=1}^{\infty} \ \frac{1}{q} \cdot T_{q} \right) + \cone(S).  \]

\end{definition}

Though Definition \ref{pBody: D} makes sense for an arbitrary semigroup, the condition that $S$ is standard will become important when establishing the basic properties of $p$-bodies.

\begin{example}
\label{constant p-system: E}
Given an arbitrary subset $T$ of $S$, setting $T_q = T+S$ for every $q$ defines a $p$-system of ideals $T_{\bullet}$ of $S$.  Indeed, each $T_q$ is clearly an ideal, and
$p T_q = pT + pS$ lies in \[ T + (p-1) T + S,\] which in turn is a subset of $T_{pq} = T+S$.  In this case, \[ \Delta_q(S, T_{\bullet}) = (1/q) \, T + \cone(S),\] and the closure of $\Delta(S, T_{\bullet})$ equals $\cone(S)$.
\end{example}

\begin{example}
\label{staircase: E}
Given an arbitrary subset $T$ of  $S$, setting $T_{q} = qT + S$ defines a $p$-system of ideals $T_{\bullet}$ of $S$ such that \[ \Delta(S, T_{\bullet}) = \Delta_{q}(S, T_{\bullet}) = T + \cone(S) \] for every $q$ a power of $p$.
\end{example}

\begin{example}
\label{sum of p-systems: E}
  Given $p$-systems $T_{\bullet}$ and $\U$ in $S$, the sequence of ideals $T_{\bullet}+\U$ whose $q$-th term is $T_q + U_q$ is also a $p$-system, and \[ \Delta_q(S, T_{\bullet} + \U) = \Delta_q(S, T_{\bullet}) + \Delta_q(S, \U) \] for every $q$ a power of $p$.  %Similarly,  the sequence $k T_{\bullet}$ obtained by scaling each term in $T_{\bullet}$ by some fixed $k \in \NN$ is also a $p$-system, and $\Delta_q(S, k T_{\bullet}) = k \cdot \Delta_q(S, T_{\bullet})$ for all $q \geq 1$. 
%The analogous relation holds for the full $p$-bodies.
\end{example}

As illustrated by Example \ref{staircase: E}, $p$-bodies need not be convex, and in applications, they are typically not. %(see \!{} for a notable exception).  
On the other hand, we see from  Definition \ref{pBody: D} that the Minkowski sum of a $p$-body $\Delta$ and the full-dimensional pointed cone generated by the ambient standard semigroup lies in $\Delta$.  As we see in the next example that this property characterizes $p$-bodies,  up to closure.

\begin{example}  Let $C$ be a full-dimensional pointed cone in $\RR^d$, and let $D$ be any subset of $C$ such that $D + C \subseteq D$.  If $S = C \cap \ZZ^d$ and  $T_q = (qD) \cap \ZZ^d$,  then the closures of ${\Delta(S, T_{\bullet})}$ and $D$ are equal:  The hypothesis that $C$ is full-dimensional implies that there exists a nonempty open ball $B$ centered at some point of $C$ with $B \subseteq C$.  After rescaling, we may assume that $B$ has radius greater than one, so that every translation of $B$ by a point in $\ZZ^d$ contains a lattice point.  Fix $\vv{u} \in D$, and for every $q$, a vector $\vv{v}_q \in B$ such that $q \vv{u} + \vv{v}_{q} \in \ZZ^d$.  By definition, $(1/q) \vv{v}_q$ lies in $(1/q)B$, which itself is contained in $C$.  Thus, 
\[ q \vv{u} + \vv{v}_q  \in q ( D + C) \cap \ZZ^d = T_q, \] 
and the fact that $\vv{v}_q \in B$ for all $q$ shows that the points 
\[ \vv{u} + (1/q) \vv{v}_q \in \Delta_q(S, T_{\bullet}) \]
 define a sequence in $\Delta(S, T_{\bullet})$ converging to $\vv{u}$.  We conclude that $D$, and hence its closure, is contained in the closure of $\Delta(S, T_{\bullet})$.  The remaining details are left to the reader.
 \end{example}
 
 \begin{remark} One may summarize the previous example and the discussion above it as follows: A closed region $\Delta$ is the closure of a $p$-body if and only if it is a $\cone(S)$-convex region.
 \end{remark}

Every $p$-body $\Delta$  is the union of countably many translates of $C$, and is contained entirely in $C$, and so is therefore Lebesgue measurable.  It follows that if $H$ is a truncating halfspace of $C$, then the volume of $\Delta \cap H$ is a well-defined real number.  The following theorem, which is concerned with volumes of this form, is the main result of this section.  

\begin{theorem}
\label{plimit} 
Let $S$ be a standard semigroup in $\ZZ^d$.  If $T_{\bullet}$ is a $p$-system in $S$, and $H$ is any truncating halfspace for $\cone(S)$, then 
\[ \lim_{q \rightarrow \infty} \frac{{\#} \( T_{q} \cap qH \)}{q^d} = \vol_{\RR^d} (\Delta(S,T_{\bullet}) \cap H). \]
\end{theorem}

The remainder of this section is dedicated to proving Theorem \ref{plimit}, and an important corollary (Corollary \ref{generalSubsetVolume: C}).  We begin with some useful lemmas, all of which hold in the context of Setup \ref{convexLemmas: SU}.

\begin{setup}
\label{convexLemmas: SU}
Let $C$ be a full-dimensional pointed cone, and $\Gamma$ be the points in $C$ lying outside some (possibly, empty) truncation of $C$.  Fix a countable subset $V$ of $C$, and set $U = V + \Gamma$.
\end{setup}

\begin{lemma}
\label{slideIntoInterior: L}  In the context of Setup \ref{convexLemmas: SU}, $\closure{U} + \interior{C}$ lies in $\interior{U}$. 
\end{lemma}

\begin{proof} As observed in Remark \ref{C-convex region: R}, $\Gamma + C$ lies in $\Gamma$.  We also note that $U + \interior{C}$ is contained in $\interior{U}$.  Indeed, given $\vv{v} \in V$ and ${\vv{w}} \in \Gamma$,
\[ \vv{v}+{\vv{w}}+\interior{C} = \vv{v} + \interior{( {\vv{w}}+C)} \subseteq \vv{v} + \interior{\Gamma}  = \interior{(\vv{v} +\Gamma)} \subseteq \interior{U}.\]  

    Choose $\vv{u} \in \closure{U}$ and $\vv{c} \in \interior{C}$.  To complete the proof, we must show that $\vv{u} + \vv{c} \in \interior{U}$. By our choice of $\vv{c}$, there exists an open ball $W$ centered at the origin such $\vv{c}+ W \subseteq \interior{C}$, and as $\vv{u} \in \closure{U}$, the open set $\vv{u} + W$ must intersect $U$.  Fix a point $\vv{z}$ in the intersection of $\vv{u} + W$ and $U$, so that $\vv{u}-\vv{z}$ is in $-W= W$.  It follows that
\[\vv{u} + \vv{c} = \vv{z} + \vv{c} + \vv{u} - \vv{z} \in \vv{z} + \vv{c} + W \subseteq \vv{z} + \interior{C} \subseteq U + \interior{C} \subseteq \interior{U},\]
which allows us to conclude the proof.\end{proof}

\begin{corollary}  
\label{measureZeroBoundary: C} 
In the context of Setup \ref{convexLemmas: SU}, $\vol_{\RR^d}(\partial U) = 0$.
\end{corollary}

\begin{proof}  As every open set and closed set in $\RR^d$ is measurable, so is the boundary $\boundary{U} = \closure{U} \setminus \interior{U}$ of $U$. Fixing $\vv{v} \in \interior{C}$, Lemma \ref{slideIntoInterior: L} implies that 
\[ \boundary{U} + \cone(\vv{v}) \subseteq \boundary{U} + \interior{C} \subseteq \closure{U} + \interior{C} \subseteq \interior{U},\] and hence, that $\( \boundary{U} + \cone(\vv{v}) \) \cap \boundary{U} = \emptyset$.

For every $n \geq 1$, let $E_n$ denote the elements of $\boundary{U}$ whose coordinate sum is less than $n$, so that the collection of sets $E_n$ as $n$ varies defines an increasing family of bounded measurable subsets whose union is $\partial{U}$.  As $\( E_n + \cone(\vv{v}) \) \cap E_n$ is empty, it follows that \[ \( E_n +(1/k) \cdot \vv{v}  \) \cap \( E_n + (1/m) \cdot\vv{v} \) = \emptyset \]  for all $k \neq m \in \NN$, and as $E_n$ is bounded,  \[ A_n := \bigcup_{k =1}^{\infty} \(E_n + (1/k) \cdot \vv{v} \) \] is also bounded, and hence, has finite measure.  The additivity of the Lebesgue measure $\mu$ on $\RR^d$ and the disjointness of the sets defining $A_n$ then imply that $\sum_{n=1}^{\infty} \mu \( E_n \)  = \mu \( A_n  \) < \infty$, which tells us that $\mu(E_n) = 0$.  Our claim then follows, as $\mu ( \partial{U} )  = \lim_{n \to \infty} \mu(E_n) = 0$.
\end{proof}

\begin{proposition}
\label{RiemannIntegral: P}
In the context of Setup \ref{convexLemmas: SU}, if $H$ is any truncating halfspace for $C$, then $D = U \cap H $ is bounded and measurable, and \[\lim_{k \to \infty} \frac{ {\#} (k D \cap \ZZ^d)}{k^d} = \vol_{\RR^d}( D ).\]
\end{proposition}

\begin{proof}  By definition, $U$ lies in $C$, and is the union of countably many translates of $C$.  Thus, $D$ is measurable and lies in the bounded truncation $C \cap H$.  Moreover, as $D$ is bounded, $\vol_{\RR^d}(D)$ can be computed via Riemann integration if and only if $\partial{D}$ has measure zero.   However, the boundary of ${D}$ is contained in the union of $\boundary{U}$ and $\boundary{H}$;  by Corollary \ref{measureZeroBoundary: C}, the former boundary has measure zero, and it is apparent that $\boundary{H}$ has measure zero in $\RR^d$ as well.  The result then follows from the observation that 
\[ \frac{\# (k D \cap \ZZ^d)}{k^d}  =   \frac{\# (D \cap (1/k) \ZZ^d)}{k^d} \]
is the volume of the Riemann approximation of $D$ given by the Minkowski sum $D \cap (1/k) \ZZ^d + \left [0, 1/k \right]^d$.
\end{proof}

We are now prepared to prove Theorem~\ref{plimit}.

\begin{proof}[Proof of Theorem \ref{plimit}]  Set $C = \cone(S)$.  To simplify notation, we write $\Delta_q$ instead of $\Delta_q(S, T_{\bullet})$ and  $\Delta$ instead of $\Delta(S, T_{\bullet})$.  If $D$ is either $\Delta$ or $\Delta_q$, then $\tilde{D}$ will denote the (bounded) intersection $D \cap H$.

By definition,  $T_q$ is contained in $q \Delta_q$.  Consequently,  
\[ T_q \cap qH \subseteq q \Delta_q \cap qH = q ( \Delta_q \cap H) = q \tilde{\Delta}_q \subseteq q \tilde{\Delta},  \] 
and Proposition \ref{RiemannIntegral: P} then implies that
 \[  \limsup_{q \to \infty} \frac{{\#}(T_q \cap qH)}{q^d} \leq \lim_{q \to \infty} \frac{{\#}( q \tilde{\Delta} \cap \ZZ^d )}{q^d} = \vol_{\RR^d}(\tilde{\Delta}).\] 

It remains to show that  the volume of $\tilde{\Delta}$ is less than or equal to the corresponding limit inferior.  Towards this, let $\mathscr{C}$ be a full-dimensional pointed cone contained in $C$ whose boundary intersects $\partial C$ only at $0$.  By Theorem \ref{ConeTheorem: T}, we may fix an a truncating halfspace $\mathscr{H}$ for $C$ such that $(\mathscr{C} \cap \ZZ^d) \setminus \mathscr{H}$ is contained in $S$.  Like $q$, let $r$ and $s$ be variable powers of $p$, and set \[M_{r,s} = \left(   \frac{1}{r} \cdot T_r + \mathscr{C} \, \setminus \,  \frac{1}{s} \cdot \mathscr{H} \right) \cap H.\] 

We claim that for every $q,r,$ and $s$, 
\begin{equation} 
\label{p-body containment: e}
qrs M_{r,s} \cap \ZZ^d \subseteq T_{qrs} \cap qrs  H.
\end{equation} 

Before justifying \eqref{p-body containment: e}, we explain how it allows us to complete the proof:  First, note that Proposition~\ref{RiemannIntegral: P} and \eqref{p-body containment: e} imply that 
\[ \vol_{\RR^d} \(  M_{r,s} \) = \frac{\vol_{\RR^d}( rs  M_{r,s})}{(rs)^d}  = \lim_{q \to \infty} \frac{ \# (qrs M_{r,s} \cap \ZZ^d)}{(qrs)^d} \]
is less than or equal to 
\[  \liminf_{q \to \infty}  \frac{\#(T_{qrs} \cap qrsH)}{(qrs)^d}  = \liminf_{q \to \infty}  \frac{\#(T_{q} \cap qH)}{q^d}. \] 

Next, observe that for a fixed value of $r$, our choice of $H$ guarantees that $((1/r) T_r) \cap H$ is finite,  which shows that as $s$ varies, the sets $M_{r,s}$ form an increasing chain whose union is
\[ M_{r} = \left(   \frac{1}{r} \cdot T_r + \mathscr{C} \right) \cap H. \]
Consequently,
\[ \vol_{\RR^d}   \( M_r \)  = \lim_{s \to \infty} \vol_{\RR^d}(M_{r,s}) \leq \liminf_{q \to \infty}  \frac{\#(T_{q} \cap qH)}{q^d}.  \] 

Similarly, as $\mathscr{C}$ approaches $C$ in an increasing manner from within, $\vol_{\RR^d}(M_r)$ approves $\vol_{\RR^d}(\tilde{\Delta}_r)$ from below, and so we may replace $\vol_{\RR^d}(M_r)$ with $\vol_{\RR^d}(\tilde{\Delta}_r)$ in the previous inequality.  Letting $r \to \infty$ then shows that $\vol_{\RR^d}( \tilde{\Delta})$ is less than or equal to the same limit inferior.

To establish \eqref{p-body containment: e}, note that since $\mathscr{C}$ is a cone, if $\lambda$ is positive and $X$ is any subset of $\RR^d$, then $\lambda (\mathscr{C} \setminus X) = (\lambda \mathscr{C}) \setminus (\lambda X) = \mathscr{C} \setminus \lambda X$. It follows from this, our choice of $\mathscr{H}$, and the fact that $T_{\bullet}$ is a $p$-system that 
\begin{align*}
qrs M_{r,s}  \cap \ZZ^d 	& = (qs T_r + \mathscr{C} \setminus qr \mathscr{H}) \cap \ZZ^d \cap qrs H  \\ 
					& \subseteq (qs T_r + \mathscr{C} \setminus \mathscr{H}) \cap \ZZ^d \cap qrs H  \\
					& = (qs T_r + (\mathscr{C} \cap \ZZ^d) \setminus \mathscr{H}) \cap qrs H \\
					& \subseteq (T_{qrs} + S) \cap qrs H \\ 
					& = T_{qrs} \cap qrsH,
\end{align*}
which allows us to conclude the proof.
\end{proof}

Below, we see that there is some flexibility when computing the volumes considered above.

\begin{lemma}  
\label{equalVolumes: L}
Let $S$ be a standard semigroup in $\ZZ^d$.  If $T_{\bullet}$ is a $p$-system in $S$, and $H$ is any truncating halfspace for $\cone(S)$,  then the volume of $\Delta(S, T_{\bullet}) \cap H$ equals the volume of its interior, and of its closure.
\end{lemma}

\begin{proof}  Set $W = \Delta(S, T_{\bullet}) \cap H$.  As in the proof of Proposition \ref{RiemannIntegral: P}, Corollary \ref{measureZeroBoundary: C} implies that $\vol_{\RR^d}(\boundary{W})=0$.  The identity $\closure{W}= \boundary{W} \sqcup \interior{W}$ then implies that $\vol_{\RR^d}(\closure{W}) = \vol_{\RR^d}(\interior{W})$.
\end{proof}

\begin{proposition}
\label{equalVolumes: P}
Let $S$ be a standard semigroup in $\ZZ^d$,  $U$  an arbitrary subset of $S$, and $\U$ the constant $p$-system given by $U_q = U+S$ for all $q$.  If $T_{\bullet}$ is a $p$-system of ideals of $S$, and $H$ is any truncating halfspace for $\cone(S)$, then the closure of $\Delta(S, T_{\bullet}) \cap H$ and the closure of $\Delta(S, T_{\bullet} + \U) \cap H$ agree.  In particular, their volumes are equal.
\end{proposition}

\begin{proof}  By Example \ref{sum of p-systems: E}, we have that 
\[ \Delta(S, T_{\bullet} + \U) \cap H = \( \Delta(S, T_{\bullet}) + \Delta(S, \U) \) \cap H,\] 
while Example \ref{constant p-system: E} shows that $\overline{\Delta(S, \U)}$ equals $C = \cone(S)$.  Using these facts, it is straightforward to verify that the closures of the stated intersections agree (the details are left to the reader).  The second statement follows from the first and Lemma \ref{equalVolumes: L}.
\end{proof}

\begin{corollary}  
\label{generalSubsetVolume: C}
Let $S$ be a standard semigroup in $\ZZ^d$, and fix an arbitrary sequence of subsets $V_{\bullet} = \{ V_q \}_{q=1}^{\infty}$ of $S$ indexed by the powers of $p$.   Let $T_{\bullet}$ be a $p$-system of ideals of $S$, and suppose that $U$ is a subset of $S$ such that
$ T_q + U  \subseteq  V_q \subseteq T_q $
for every $q$.   
If $H$ is any truncating halfspace of $\cone(S)$, then
\[  \lim_{q \to \infty}  \frac{ \# (V_q \cap qH)}{q^d} =  \vol_{\RR^d} (\Delta(S,T_{\bullet}) \cap H). \]
\end{corollary}

\begin{proof}  Let $\U$ be the constant $p$-system of ideals of $S$ determined by $U$.  By hypothesis,  $T_q +U_q \subseteq V_q \subseteq T_q$, and therefore  \[  \#\(\( T_q + U_q \) \cap qH\) \leq \#(V_q \cap qH) \leq \#(T_q \cap qH). \] The result then follows from Theorem \ref{plimit}  and Proposition \ref{equalVolumes: P}.
\end{proof}

\section{$p$-families, $p$-systems, and limits as Euclidean volumes}\label{limits-as-volumes}
\newcommand{\I}{I_{\bullet}}
\newcommand{\J}{J_{\bullet}}

In this sections, all rings will be of prime characteristic $p>0$.  Recall the following definitions, which appeared in the introduction.

\begin{definition}
A sequence of ideals $\I = \{ I_q \}_{q=1}^{\infty}$ whose terms are indexed by the powers of the characteristic is called a \emph{$p$-family} whenever $I_q^{[p]} \subseteq I_{pq}$ for all $q$ a power of $p$. \end{definition}

\begin{definition}
Given a $p$-family $\I$ of $\m$-primary ideals in a $d$-dimensional local ring $(R, \m)$, we call the limit 
\[ \vol_R( \I) = \lim_{q \to \infty} \frac{\ell_R ( R / I_q )}{q^d} \] the \emph{volume} of the $p$-family.  
\end{definition}

\begin{remark}
\label{uniform multiplier: R}
Given a $p$-family of ideals $\I$ in a local ring $(R, \m)$, then the terms in this family are $\m$-primary if and only if $I_1$ is $\m$-primary.  Indeed, if $\m^{a}$ is contained in $I_1$ for some positive integer $a$, and $I_1$ is generated by $b$ elements, then, by the pigeon-hole principle, 
%
%
%\begin{equation}
%\label{uniform multiplier: e}
\[ \m^{ ab q} \subseteq I_1^{b q} \subseteq I_1^{[q]} \subseteq I_{q}. \]
%\end{equation}
In other words, if $c = ab$, then $\m^{cq} \subseteq I_q$ for all $q$ a power of $p$.  
\end{remark}

Perhaps the most basic example of a $p$-family is the sequence whose $q$-term is the $q$-th Frobenius power of some fixed ideal in a ring of characteristic $p>0$.  Below, we consider others.

\begin{example}%[$p$-families from certain $R$-linear maps]
\label{F-sig family: E}
Recall that if $e \geq 0 $ is an integer, then $F^e_{\ast} R$ is the $R$-module obtained by applying the restriction of scalars functor associated to the $e$-th iterated Frobenius endomorphism of $R$ to the ring $R$ itself.  In concrete terms, an $R$-linear map $\phi: F^e_{\ast} R \to R$ is simply a map of sets $\phi: R  \to R$ that is additive, and such that $\phi(x^{p^e} y) = x \phi(y)$ for every $x$ and $y$ in $R$.   With this notation, if $\mathfrak{b}$ is an ideal of $R$, and $I_q$ consists of all $x \in R$ such that $\phi(x) \in \mathfrak{b}$ for every $R$-linear map $\phi: F^e_{\ast} R \to R$, then the corresponding sequence $\I$ is a $p$-family of ideals in $R$.  For a detailed verification of this fact when $\mathfrak{b}$ is the maximal ideal of a local ring, we refer the reader to \cite{Tucker12}.
\end{example}

\begin{example}%[$p$-families from graded families]  
Recall that a \emph{graded family of ideals} is a sequence of ideals $\J = \{ J_n \}_{n =1}^{\infty}$ indexed by the natural numbers with the property that $J_a \cdot J_b \subseteq J_{a+b}$ for all natural numbers $a$ and $b$.  Given a graded system of ideals, the subsequence corresponding to terms indexed by powers of $p$ defines a $p$-family.    For example, if $\mathfrak{b}$ is an ideal of $R$, then the sequence whose $n$-th term is $\mathfrak{b}^n$ defines a graded family of ideals, from which we may extract the $p$-family whose $q$-th term is $\mathfrak{b}^q$.
\end{example}

It is apparent that one may employ standard algebraic constructions to generate new $p$-bodies from given ones.  We gather some instances of this below.

\begin{example}%[New families from old]
The termwise product, sum, or intersection of an arbitrary collection of $p$-families defines a $p$-family.  Moreover, given a map of rings $A \to B$, the termwise expansion of a $p$-family in $A$ defines a $p$-family in $B$.  Similarly, the termwise contraction of a $p$-family in $B$ defines a $p$-family in $A$.
The termwise saturation of a $p$-family is also a $p$-family.  More precisely, if $\I$ is a $p$-family in $R$, and $\mathfrak{b}$ is any ideal of $R$, then the sequence whose $q$-term is \[ (I_q : \mathfrak{b}^{\infty}) = \cup_{n=1}^{\infty} (I_q: \mathfrak{b}^{n})\] is a $p$-family in $R$.
\end{example}

\begin{example}%[$p$-families involving parameters]
\label{parameter family: E}
Fix a positive real number $\lambda$.  If $q$ is a power of $p$,  set $\lambda_q = 
\lceil  q \lambda \rceil -1$, and note that $\lambda_{pq}$ is greater than or equal to $p \lambda_q$. 

Fix an element $f$ of a ring $R$, and a $p$-family $\I$ in $R$.  If  \[ J_q = (I_q : f^{\lambda_q})\] for every $q$ a power of $p$, then
\[ J_q^{[p]} \subseteq (I_q^{[p]} : f^{p \lambda_q}) \subseteq (I_{pq} : f^{\lambda_{pq}}) = J_{pq}, \] which shows that $\J$ is a $p$-family of ideals in $R$.
\end{example}

\begin{setup}
\label{OK: SU}
Fix a $d$-dimensional local domain $(D, \m, \kk)$ of characteristic $p>0$ with fraction field $\FF$, a  $\ZZ$-linear embedding of $\ZZ^d$ into $\RR$ induced by a vector $\vv{a}$ in $\RR^d$, and a valuation $\nu: \FF^{\times} \longonto \ZZ^d$ that is OK relative to $D$.   We use $S$ to denote the semigroup $\nu(D)$ in $\ZZ^d$, and $C$ to denote the closed cone in $\RR^d$ generated by $S$.  We follow the notation established in Definition \ref{OKvaluation: D}.  
\end{setup}

The connection between $p$-families and $p$-systems stems from the simple fact that if $x \in D$, then $p \cdot \nu(x) = \nu(x^p)$.  Consequently, if $\I$ is a $p$-family of ideals in $D$, then $\T(\I)$, the sequence whose $q$-th term is $\T(I_q)$, is a $p$-system of ideals in $S$.  

\begin{remark} Whether the sequence $\Th(\I)$ is a $p$-system when $h > 1$ appears to be a more subtle issue.  Indeed, suppose that $M$ is a $D$-submodule of the fraction field $\FF$, and let $m_1, \dots, m_h$ be elements of $M$ with $\nu(m_1) = \cdots = \nu(m_h) = \vv{u}$, and whose images in 
\[ \frac{ M \cap \FF_{\geq \vv{u}}}{M \cap \FF_{> \vv{u}}} \] 
are linearly independent over $\kk$.  If the images of $m_1^p, \dots, m_h^p$ in 
\[ \frac{ M \cap \FF_{\geq p\vv{u}}}{M \cap \FF_{> p\vv{u}}} \] 
were always linearly independent over $\kk$, then this would tell us that $\Th(\I)$ was a $p$-system.    Though this is condition {is} satisfied whenever $\kk$ is perfect, it is not clear whether $\Th(\I)$ is a $p$-system in general, and it is possible that this issue may depend on the particular  valuation.
\end{remark}

While we do not know whether the sequence $\Th(\I)$ is a $p$-system when $h>1$, the fact that it can be uniformly approximated by $\T(\I)$ allows us to regard it as such in at least one important way.

\begin{corollary}  
\label{limitViaVolume: C}
Adopt the context of Setup \ref{OK: SU}. If $\I$ is a $p$-family of ideals in $D$ and $h$ is an integer with $1 \leq h \leq [\kk_V: \kk]$, then
\[ \lim_{q \to \infty} \frac{ \# ( \Th(I_q)  \cap qH ) }{q^d}  = \vol_{\RR^d}(\Delta(S, \T(\I)) \cap H) , \]
where $H$ is any truncating halfspace for the cone $C$.
\end{corollary}

\begin{proof}  This follows immediately from Lemma \ref{approximation: L} and Corollary \ref{generalSubsetVolume: C}.
\end{proof}

\begin{theorem}
\label{OK limit existence: T}
  Adopt the context of Setup \ref{OK: SU}, and fix $p$-families of ideals $\I$ and $\J$ in $D$ with $I_q \subseteq J_q$ for all $q$.   If there exists a positive integer $c$ such that $\m^{cq} \cap I_q = \m^{cq} \cap J_q$ for all $q$, then 
\[ \lim_{q \to \infty} \frac{ \ell_D \( J_q / I_q \) }{q^d} =  [\kk_V: \kk] \cdot \vol_{\RR^d} \( \Delta(S, \T(\J)) \setminus \Delta(S, \T(\I)) \). \] 

Moreover, the $p$-bodies $\Delta(S, \T(\I))$ and $\Delta(S, \T(\J))$ agree outside of some truncation of the cone $C$, and so the the above limit is finite.
\end{theorem}

Before proving Theorem \ref{OK limit existence: T}, we observe an immediate application.  Let $\J$ be the constant $p$-family whose terms are all $D$.  In this case, $\Delta(S, \nu(\J)) = C$, and if $\I$ is any $p$-family of $\m$-primary ideals in $D$, then Remark \ref{uniform multiplier: R} tells us that $\I$ and $\J$ satisfy the hypotheses of Theorem \ref{OK limit existence: T}.  Thus, we obtain the following corollary.

\begin{corollary} 
\label{OK limit of m-primary ideals exists: E}
Adopt the context of Setup \ref{OK: SU}.  If $\I$ is a $p$-family of $\m$-primary ideals in $D$, and $\Delta = \Delta(S, \T(\I))$, then 
\[ \vol_D(\I) = [ \kk_V : \kk ] \cdot \vol_{\RR^d} ( C \setminus \Delta ) < \infty.\] 
\end{corollary}

\begin{remark} Let $R$ be an $\NN$-graded ring whose zeroth component is a field, and $\I$ and $\J$ be $p$-families of homogeneous ideals with $I_q \subseteq J_q$ for all $q$. In this setting, the condition that there exists a positive integer $c$ such that $\m^{cq} \cap I_q = \m^{cq} \cap J_q$ for all $q$ is equivalent to the condition that there exists some $b$ such that $[I_q]_{\geq bq}=[J_q]_{\geq bq}$ for all $q$. Indeed, if $R$ is generated in degrees at most $d$, then $[R]_{\geq cdq} \subseteq \idealm^{cq} \subseteq [R]_{\geq cq}$. Thus, if $[I_q]_{\geq cq}=[J_q]_{\geq cq}$ for all $q$, then 
\[ \idealm^{cq} \cap  I_q= \idealm^{cq} \cap  [I_q]_{\geq cq}= \idealm^{cq} \cap [J_q]_{\geq cq}  = \idealm^{cq} \cap J_q .\]
Likewise, if $\m^{cq} \cap I_q = \m^{cq} \cap J_q$ for all $q$, then
\[ [I_q]_{\geq cdq} = \m^{cq} \cap I_q \cap [R]_{\geq cdq} = \m^{cq} \cap J_q \cap [R]_{\geq cdq} = [I_q]_{\geq cdq}.\]
\end{remark}

A more in-depth discussion of the condition on the $p$-families $\I$ and $\J$ appearing in the statement of Theorem \ref{OK limit existence: T} in an important special case can be found in Section \ref{applications}.

\begin{proof}[Proof of Theorem \ref{OK limit existence: T}]   If $\vv{u} \in \ZZ^d$, then the exact sequences 
\[ 0  \to \frac{ I_q }{ I_q \cap \FF_{\geq \vv{u} } } \to \frac{ J_q }{ I_q \cap \FF_{\geq \vv{u} } } \to \frac{ J_q }{ I_q  } \to 0 \] 
and 
\[ 0  \to \frac{ J_q \cap \FF_{\geq \vv{u} } }{ I_q \cap \FF_{\geq \vv{u} } } \to \frac{ J_q }{ I_q \cap \FF_{\geq \vv{u} } } \to \frac{ J_q }{ J_q \cap \FF_{\geq \vv{u} } } \to 0 \vspace{2mm} \] 
show that the length of $J_q / I_q$ equals 
\begin{equation} 
\label{length relations: e}
\ell_R \(  \frac{ J_q \cap \FF_{\geq \vv{u} } }{ I_q \cap \FF_{\geq \vv{u} } }  \) + \ell_R \(  \frac{ J_q }{ J_q \cap \FF_{\geq \vv{u} } }  \) - \ell_R \( \frac{ I_q }{ I_q \cap \FF_{\geq \vv{u} } }   \).
\end{equation}

Let $\vv{v} \in S$ satisfy the last condition in Definition \ref{OKvaluation: D} relative to  $\nu$.  If $\vv{w} = c \vv{v}$, then $D \cap \FF_{\geq q \vv{w} } = D \cap \FF_{\geq cq \vv{v} }  \subseteq  \m^{cq}$, and therefore 
\begin{equation}  
\label{uniform intersection: e}
I_q \cap \FF_{\geq q \vv{w} }= J_q \cap \FF_{\geq q \vv{w}}  \text{ for all $q$}. 
\end{equation}

If $H$ is the halfspace of $\RR^d$ consisting of all $\vv{u}$ with $\iprod{\vv{u}}{\vv{a}} < \iprod{\vv{u}}{\vv{w}}$,  then \eqref{uniform intersection: e} implies that the $p$-bodies associated to $\T(\I)$ and $\T(\J)$ agree outside of the truncated cone $C \cap H$.  Furthermore, \eqref{uniform intersection: e}, \eqref{length relations: e}, and Lemma \ref{lengths via counting points: L} imply that 
\begin{align*}
\ell_D(J_q/I_q) & = \ell_D \( {J_q} / {J_q \cap \FF_{\nu \geq q \vv{w}}} \)  -  \ell_D \( {I_q} / {I_q \cap \FF_{\nu \geq q \vv{w}}} \)  \\ 
& = \sum_{h=1}^{[\kk_V: \kk]}  \# \( \Th(J_q) \cap qH \) - \sum_{h=1}^{[\kk_V: \kk]}  \# \( \Th(I_q) \cap qH \). 
\end{align*} 

It then follows from Corollary  \ref{limitViaVolume: C} that $\lim \limits_{q \to \infty} \ell_D( J_q/I_q) / q^d$ equals
\begin{align*}
&  [\kk_V: \kk] \cdot (\vol_{\RR^d}(\Delta(S, \T(\J)) \cap H) - \vol_{\RR^d}(\Delta(S, \T(\I)) \cap H))  \\ 
= \ &  [\kk_V: \kk] \cdot  \vol_{\RR^d}((\Delta(S, \T(\J)) \setminus \Delta(S, \T(\I))) \cap H), 
\end{align*}
which by an earlier observation equals $[\kk_V: \kk]$ times the volume of the difference of $p$-bodies $\Delta(S, \T(\J)) \setminus \Delta(S, \T(\I))$.
\end{proof}

We conclude this subsection with an extension of Theorem \ref{OK limit existence: T}.   We also note that a refinement of this result in the special case that $\J$ is the constant $p$-family whose terms are all given by the ambient ring appears in Subsection~\ref{volume formulas: SS}.

\begin{theorem}
\label{limits-of-pairs-reduced}
Let $(R,\m)$ be a reduced local ring of characteristic $p>0$ such that the quotient of $R$ by each minimal prime is an OK domain (e.g., $R$ is a reduced complete local ring of characteristic $p>0$).  Fix  $p$-families of ideals $\I$ and $\J$ in $R$ with $I_q \subseteq J_q$ for all $q$.   If there exists a positive integer $c$ such that $\m^{cq} \cap I_q = \m^{cq} \cap J_q$ for all $q$, then 
\[ \lim_{q \to \infty} \frac{ \ell_D \( J_q / I_q \) }{q^d} \]
exists.
\end{theorem}
\begin{proof} This follows immediately from Theorem~\ref{OK limit existence: T} and Proposition~\ref{pairs-to-reduced}, which appears below.
\end{proof}

\subsection{A characterization of when limits exist}

In this subsection, we derive a characterization of when volumes of $p$-families exist in general by reducing to the case of an OK domain.  For what follows, recall that the dimension an $R$-module $M$ is the Krull dimension of $R$ modulo the annihilator in $R$ of $M$.

\begin{theorem} 
\label{m-primary limits exist: T}
 Let $(R, \m, \kk)$ be a local ring of characteristic $p>0$ and dimension $d$.  If $R$-module dimension of the nilradical of $\hat{R}$ is less than $d$, then for any $p$-family of $\m$-primary ideals $\I$ of $R$, the limit \[ \vol_R(\I) = \lim_{q \to \infty} \frac{ \ell_R( R/ I_q) }{q^{d}} \] exists and is finite.  

Conversely, if this limit exists for all $p$-families of $\m$-primary ideals, then the $R$-module dimension of the nilradical of $\hat{R}$ is less than $d$.
\end{theorem}

In Example \ref{volume DNE: E} from the Introduction, we saw that the condition on the nilradical of $\hat{R}$ in Theorem \ref{m-primary limits exist: T} is necessary, and  more generally, the proof of \cite[Theorem~5.4]{Cutkosky14}, due to Dao and Smirnov,
may be slightly modified to show when the dimension of the nilradical is maximal, there exists a $p$-family whose volume does not exist. We reproduce the adapted proof below to keep the exposition self-contained (and to avoid any confusion related to a few small typos in \emph{loc.\,cit.}).

\begin{proposition}
\label{limit DNE: P}
Let $(R, \m, \kk)$ be a local ring of characteristic $p>0$ and dimension $d$.  If the $R$-module dimension of the nilradical of $\hat{R}$ equals $d$, then there exists a $p$-family of ideals $\I$ in $R$ for which \[ \vol_R(\I) = \lim_{q \to \infty} \frac{\ell_R(R/I_q)}{q^d}\] does not exist. 
\end{proposition}

\begin{proof}  
Note that if a $p$-family satisfying the conclusion of the proposition exists in the completion $(\hat{R}, \idealm_{\hat{R}})$, then one exists in $R$ itself. Namely, if $\I$ is a $p$-family of $\idealm_{\hat{R}}$-primary ideals in $\hat{R}$, then the intersection of $\I$ with $R$ is a $p$-family of $\idealm$-primary ideals in ${R}$, and \[ \ell_{\hat{R}}(\hat{R}/I_q)=\ell_{{R}}({R}/(I_q \cap R)),\] which is immediate from the isomorphisms $R/\idealm^t \cong \hat{R}/\idealm^t_{\hat{R}}$ for every non-negative integer $t$ induced by the inclusion $R \longhookrightarrow \hat{R}$. Thus, we assume that $R$ is complete with $d$-dimensional nilradical.

 If $N$ is the nilradical of $R$, and $\mathfrak{a}$ its annihilator, then there exists a minimal prime $\mathfrak{p}$ of $\mathfrak{a}$ such that $\dim(R/\mathfrak{p}) = d$.  This implies that $\mathfrak{p}$ is a minimal (and hence, associated) prime of $R$, and so there exists a nonzero $x \in R$ such that $\mathfrak{p} = (0: x)$.  If $x \notin \mathfrak{p}$, then 
\[  0 = x \mathfrak{p} R_{\mathfrak{p}} = \mathfrak{p} R_{\mathfrak{p}} = N R_{\mathfrak{p}},\] which is impossible, since $\mathfrak{p}$ contains $\mathfrak{a}$, the annihilator of $N$, by assumption.  For what follows, we fix $x$ and $\mathfrak{p}$, and we note that $x^2 = 0$.

Let $b_q$ be any sequence of integers  indexed by the powers of $p$ such that $\lim_{q \to \infty} b_q = \lim_{q \to \infty} (q-b_q) = \infty$, and such that $\lim_{q \to \infty} \( b_q/q \)^d$ does not exist.  For example, one could take the sequence whose $p^e$-th term is $\lfloor p^e/3 \rfloor$ if $e$ is odd, and $\lfloor 2p^e/3 \rfloor$ if $e$ is even.   For every $q$, set \[ I_q=\m^q + x \m^{b_q}.\]  The fact that $x^2=0$ implies that $\I$ is a $p$-system of ideals of $R$.

%\!{Let $b_q$ be the sequence (indexed by powers $q$ of $p$) such that $b_q=\lfloor q/3 \rfloor$ if $\log_p(q)$ is odd and $b_q=\lfloor 2q/3 \rfloor$ if $\log_p(q)$ is even. Note that $\lim_{q\rightarrow\infty} (q-b_q) =\infty$ and $\lim_{q\rightarrow\infty} b_q/q$ does not exist.
%
%As in \emph{loc.\,cit.}, the hypothesis on $R$ guarantees the existence of a prime $\idealp$ and element $x$ such that with $\dim(R/\idealp)=d$, $\idealp=\mathrm{ann}(x)$, and $x^2=0$. Set $I_q=\m^q + x \m^{b_q}$. It follows from the fact that $x^2=0$ that $I_q$ form a $p$-system of ideals. }

Let $(A, \idealm_{A})$ be the local ring $R/xR$.  From the short exact sequence
\[ 0 \rightarrow xR / (xR \cap I_q) \rightarrow R/I_q \rightarrow A / I_q A \rightarrow 0\]
we have that 
\[ \frac{\ell_R(R/I_q)}{q^d}  =  \frac{\ell_R(xR / (xR \cap I_q))}{q^d} + \frac{\ell_A(A / I_q A)}{q^d}. \]
  
We claim that the second expression on the right hand side converges as $q\rightarrow \infty$ while the first expression on the right hand side diverges.  The first claim is evident, since $I_q A = \m_{A}^q$, so the sequence converges to a rational multiple of the multiplicity of the $d$-dimensional ring $A$.

The Artin-Rees lemma implies that for some $k>0$ and all $q \geq k$,
\[ xR\cap \m^q = \m^{q-k}(xR\cap \m^k) \subseteq x \m^{q-k} \subseteq x \m^{b_q}. \]
 Thus, $xR \cap I_q = x \m^{b_q}$ for sufficiently large~$q$.  If $(D, \m_D)$ is the local ring $R / \mathfrak{p}$, then this observation and the $R$-linear isomorphism $R/\mathfrak{p} \cong xR$ show that for large $q$ 
 \[ \frac{xR}{xR\cap I_q} \cong \frac{xR}{x\m^{b_q}}
 \cong  \frac{D}{\m_D^{b_q}}.\]
If $H(t) = \ell_D( R/ \m_D^t)$ is the Hilbert function of $D$ and $q$ is large, then
\[  \frac{\ell_R(xR / (xR \cap I_q))}{q^d}  = \frac{\ell_D(D/\m_D^{b_q})}{q^d} = \frac{H(b_q)}{q^d}, \] 
which limits to a rational multiple of $\lim_{q \to \infty}(b_q/q)^d$, and therefore does not exist.   The claim is established, and the Proposition follows.
\end{proof}

We now turn our attention to proving Theorem~\ref{m-primary limits exist: T}.  In light of Proposition \ref{limit DNE: P}, it is only necessary to justify the first statement.  To do so, we will rely on a handful of results that appear in the next subsection.

\begin{proof}[Proof of Theorem \ref{m-primary limits exist: T}]

The length of $R/I_q$ is unaffected by completion, and so we may assume that $R$ is complete.  If $N$ is the nilradical of $R$ and $A=R/N$, then Corollary \ref{reduced} and our assumption on the dimension of $N$ imply that the limit in question exists if and only if \[ \lim_{q \to \infty} \frac{\ell_A (A/I_qA)}{q^d} \] exists.  Thus, we may also assume that $R$ is reduced.  Moreover, %Corollary \ref{domain}
Remark \ref{uniform multiplier: R} and Proposition \ref{pairs-to-reduced} allow us to further  assume that $R$ is a domain, and therefore, OK by Corollary \ref{complete local domain is OK: C}.  The claim then follows from Corollary~\ref{OK limit of m-primary ideals exists: E}.
\end{proof}

%\!{To get to the domain case from the reduced one, we can instead appeal to \eqref{uniform multiplier: e}and Proposition \ref{pairs-to-reduced}.  This allows us to remove Corollary 5.18.  Actually, that whole thing is true if we really can get rid of the assumption in Theorem 5.13 that the $p$-families are nested decreasing in Proposition 5.15 (where it doesn't appear, but should if we can't accomplish this) and its corollary Theorem 5.13}

\subsection{Some simplifications}

In this subsection, we gather the results needed to reduce Theorem \ref{limits-of-pairs-reduced} and Theorem \ref{m-primary limits exist: T} to the case of an OK domain.   
Note that Corollary \ref{reduced} below is an adaption of \cite[Theorem~4.7]{Cutkosky13}, %Corollary \ref{domain} is an adaption of \cite[Lemma~5.1]{Cutkosky13}, 
and Proposition~\ref{pairs-to-reduced} is an adaptation of the second half of the proof of \cite[Theorem~11.1]{Cutkosky13}.

\begin{lemma}
\label{basicBound: L}
Let $(R, \m)$ be a $d$-dimensional local ring of characteristic $p>0$, and $\I$ be a sequence of ideals of $R$ indexed by the powers of $p$ such that $\m^{cq} \subseteq I_q$ for some positive integer $c$ and all $q$ a power of $p$.  
If $M$ is a finitely generated $R$-module, then there exists $\alpha>0$ such that
$\ell_R(M/I_qM) \leq \alpha \cdot q^{\dim M}$ for every $q$.
\end{lemma}

\begin{proof} The length of $M/I_qM$ is less than or equal to that of $M/\m^{cq} M$, and as the latter length agrees with a polynomial in $cq$ of degree $\dim(M)$ for all but finitely many values of $q$, we may choose a  positive constant $\beta$ such that $\ell_R(M/\m^{cq}M)$ is bounded above by $\beta(cq)^{\dim M} = (\beta c^{\dim M}) \cdot q^{\dim M}$ for all values of $q$.
\end{proof}

\begin{corollary}
\label{reduced}
Let $(R, \m)$ and $\I$ be as in Lemma \ref{basicBound: L}.  If $N$ is an ideal of $R$ and $A=R/N$, then there exists $\beta > 0$ such that  \[ 0 \leq \ell_R(R/I_q) - \ell_A(A/I_qA) \leq \beta \cdot q^{\dim N}\] for all $q$ a power of $p$.  
\end{corollary}

\begin{proof} Note that $\ell_A(A/I_qA) = \ell_R(A/I_qA)$, and therefore the short exact sequence $0 \to N/N \cap I_q \to R/I_q \to A/I_qA \to 0$ implies that \[ \ell_R(R/I_q) - \ell_A(A/I_qA) = \ell_R(N/N \cap I_q).\]  

However, as $I_q N \subseteq I_q \cap N$,  Lemma \ref{basicBound: L} then implies that there exists $\beta > 0$ such that $\ell_R(N/I_q \cap N) \leq \ell_R(N/I_qN) \leq \beta  q^{\dim N}$.
\end{proof}

\begin{corollary}
\label{flexibility F-sig of pairs: C}

Let $(R, \m)$ and $\I$ be as in Lemma \ref{basicBound: L}.  If  $L_{\bullet}$ is a sequence of ideals in $R$ indexed by the powers of $p$ such that \[ I_q \subseteq L_q \subseteq (I_q: x)\] for some nonzerodivisor $x$ in $R$ and for every $q$ a power of $p$, then \[ \ell( R/I_q)  - \ell ( R/L_q) \leq \gamma \cdot q^{d-1} \] for some positive constant $\alpha$ and for all $q$ a power of $p$.

\end{corollary}

\begin{proof}  If $A = R/xR$, then the exact sequence
\[ 0 \to L_q/ I_q \to R/I_q \to R/L_q \to 0, \] 
the injection $L_q /I_q  \longhookrightarrow (I_q : x)/I_q$, and the exact sequence
\[ 0 \to (I_q : x)/ I_q \to R/I_q \stackrel{x}{\longrightarrow} R/I_q \to A/ I_q A \to 0, \]
imply that the difference in question is bounded by $\ell(A/I_qA)$.  The Corollary then follows from Lemma \ref{basicBound: L}.
\end{proof}

\begin{proposition}
\label{pairs-to-reduced} Consider the following property of a local ring. \ \\

\noindent ($\star$) For every pair of $p$-families of ideals $I_\bullet$ and $J_\bullet$ in a local ring $(A, \idealm_A)$ of characteristic $p>0$ such that $I_q \subseteq J_q$ for every $q$, and such that $\idealm_A^{cq} \cap J_q = \idealm_A^{cq} \cap I_q$ for some positive integer $c$ and for every $q$, the limit \[ \lim_{q\rightarrow\infty} \frac{\ell_A(J_q / I_q)}{q^{\dim(A)}} \] exists.\\

 If $(R, \m_R)$ is reduced, and the quotient of $R$ by each of its minimal primes satisfies the property ($\star$), then $R$ satisfies the property ($\star$) itself.
\end{proposition}

\begin{proof}Suppose that $\I$ and $\J$ are $p$-bodies in $R$ with $I_q \subseteq J_q$ for all $q$, and such that $\m_R^{bq} \cap I_q = \m_R^{bq} \cap J_q$ for some positive integer $b$ and all $q$.

Let $\idealp_1,\dots,\idealp_m$ be the minimal primes of $R$, let $(R_i, \idealm_{R_i})$ be the quotient of $R$ by $\idealp_i$, and let $T = \oplus_{i=1}^m R_i$. For every integer $n \geq 0$, set  
\[ \omega_n=\idealm_{R}^n T \cap R = (\idealm_R^n + \idealp_1) \cap \cdots \cap (\idealm_R^n + \idealp_m). \]

By the Artin-Rees lemma, there exists an integer $k \geq 0$ such that \[ \omega_n  = \m_R^{n-k} \omega_k \subseteq \m_R^{n-k} \] for all $n \geq k$.  Thus, if $c = b+1$ and $q \geq k$, it follows that $\omega_{cq}$ is contained in $\m^{bq}$, so that $\omega_{cq} \cap I_q = \omega_{cq} \cap J_q$, and therefore 
\[\ell_R(J_q/I_q) = \ell_R(J_q / (\omega_{c q} \cap J_q)) - \ell_R(I_q / (\omega_{c q}\cap I_q ))\]
for all $q$ large. Thus, we may assume that $I_q$ is of the form $\omega_{cq} \cap J_q$ for some positive integer~$c$.

Now, define $p$-families of ideals $L^{(0)}_\bullet, L^{(1)}_{\bullet},\dots, L^{(m)}_{\bullet}$ by $L^{(0)}_q = J_q$ and 
\[ L^{(i)}_q = J_q \cap (\m_R^{cq} + \idealp_1) \cap \cdots \cap (\m_R^{cq} + \idealp_i) \] for all $1 \leq i \leq m$.  In particular, $L^{(m)}_q=\omega_{cq} \cap J_q$, and so 
\[ \ell_R( J_q / \omega_{cq} \cap J_q) =  \ell_R(L^{(0)}_q / L^{(m)}_q) = \sum_{i=1}^m \ell_R (L^{(i-1)}_q / L^{(i)}_q). \]

Therefore, it suffices to show that 
\[\lim_{q\rightarrow \infty} \frac{\ell_R( L^{(i-1)}_q / L^{(i)}_q )}{ q^d}\] exists for all $1 \leq i \leq m$. However,  the quotient map $L^{(i-1)}_q \twoheadrightarrow L^{(i-1)}_q R_i$ and the recursion $L_q^{(i)} = L_q^{(i-1)} \cap (\m_R^{cq} + \idealp_i)$ induce isomorphisms

\[\frac{L^{(i-1)}_q} { L^{(i)}_q}  = \frac{L^{(i-1)}_q} { L^{(i-1)}_q \cap (\idealm_R^{cq} + \idealp_i)}  \cong \frac{L^{(i-1)}_q R_i} { L^{(i-1)}_q R_i \cap \idealm_{R_i}^{cq}},\] and it is apparent that the 
sequences of ideals in the numerator and denominator of the right-most term above are $p$-families in the domain $R_i$ that satisfy the condition in the statement of the property $(\star)$. As the Krull dimension of each $R_i$ is bounded above by the Krull dimension of $R$, we conclude that
\[\lim_{q\rightarrow \infty} \frac{\ell_R( L^{(i-1)}_q / L^{(i)}_q )}{ q^d} = \lim_{q\rightarrow \infty} \frac{\ell_{R_i} (L^{(i-1)}_q R_i / (L^{(i-1)}_q R_i \cap \idealm_{R_i}^{cq}))}{q^d}\] exists, which is what we needed to show.
\end{proof}

\subsection{A refined description of volumes}
\label{volume formulas: SS}
In this subsection, we derive a refinement of an important special case of Theorem \ref{limits-of-pairs-reduced}.  We begin with a useful lemma.

\begin{lemma}
\label{domain}
Let $(R, \m)$ be a $d$-dimensional local ring of characteristic $p>0$, and $\I$ be a sequence of ideals of $R$ indexed by the powers of $p$ such that $\m^{cq} \subseteq I_q$ for some positive integer $c$ and all $q$ a power of $p$.  Suppose $R$ is reduced, and that $\idealp_1, \dots, \idealp_n$ are the minimal primes of $R$.  If $R_i= R/\idealp_i$, then there exists $\delta > 0$ such that for all $q$,  
\[ \left| \sum_{i=1}^n \ell_{R_i}( R_i / I_q R_i) - \ell_R(R/I_q) \right| \leq \delta \cdot q^{d-1}. \]  
\end{lemma}

\begin{proof}  As $R$ is reduced, we have embeddings
\[  R \longhookrightarrow A:=\prod_{i=1}^n R_i \longhookrightarrow B:= \prod_{i=1}^n \mathbb{F}_i, \]
where $\mathbb{F}_i$ is the fraction field of $R_i$.  It is well-known that $B$ is isomorphic to the localization of $R$ at the multiplicative set consisting of its nonzerodivisors (i.e., the ring of total fractions of $R$).  As $A$ is a finitely-generated $R$-module, it follows that there exists an element $x \in R$ that is a nonzerodivisor on both $A$ and $R$, and such that $x A \subseteq R$.  

Next, observe that $\ell_R(A/I_q A) = \sum_{i=1}^n \ell_{R_i} (R_i/I_q R_i)$, and so it suffices to bound $| \ell_R(A/I_qA) - \ell_R(R/I_q)|$ from above.  However, the inclusion $R \longhookrightarrow A$ gives rise to an exact sequence of $R$-modules
 \begin{equation}
\label{basicSES: e}
 0 \to (I_q A \cap R)/I_q \to R/I_q \to A/I_q A \to C_q \to 0,
\end{equation}
 where $C_q \cong A/(R + I_qA)$.  In light of this, it suffices to bound the $R$-module lengths of $(I_q A \cap R )/ I_q )$ and $C_q$.

We begin by considering the first length.  By our choice of $x \in R$, we have that $I_q \subseteq I_qA \cap R \subseteq (I_q: x)$, and Corollary \ref{flexibility F-sig of pairs: C} then tells us that
\[ \ell_R((I_q A \cap R)/I_q) \leq \alpha \cdot q^{d-1}\]
for some positive constant $\alpha$ and for all values of $q$.
  
We now focus our attention on $\ell_R(C_q)$.  Set $\bar{A} = A/xA$.  By our choice of $x$, we have that $I_q A + xA \subseteq I_q A + R$, and therefore, 
\[ \bar{A}/I_q \bar{A} \cong A/(I_qA + xA) \twoheadrightarrow A/(I_qA + R) \cong C_q.\] 
It follows from this and Lemma \ref{basicBound: L} that there exists $\beta > 0$ such that
\[ \ell_R(C_q) \leq \ell_R(\bar{A}/I_q \bar{A}) \leq \beta q^{\dim{\bar{A}}} = \beta \cdot q^{d-1}.\]  
for all values of $q$.

In summary, \eqref{basicSES: e} and the bounds established above show that 
\[ \left| \ell_R(A/I_qA) - \ell_R(R/I_qR) \right| \leq \max\{\alpha,\beta\} \cdot q^{d-1}\] for all $q$,  which allows us to conclude the proof.
\end{proof}

\begin{proposition}
\label{volume-as-sum: P}
Let $(R,\m)$ be a reduced local ring of characteristic $p>0$ such that the quotient of $R$ by each minimal prime is an OK domain (e.g., $R$ is a reduced complete local ring of characteristic $p>0$).  If $\I$ is a $p$-family of $\m$-primary ideals in $R$, then \[ \vol_R(\I) = \sum \vol_A(\I \hspace{.5mm} A ),\] where the sum above is over all quotients $A=R/ \idealp$ of $R$ by a minimal prime $\idealp$ such that $\dim(R) = \dim(R/\idealp)$.  
\end{proposition}

\begin{proof}
By Theorem \ref{limits-of-pairs-reduced}, we know that the volume of $\I$ in $R$ and the volume of $\I$ modulo any minimal prime of $R$ exist.  If $\idealp_1, \dots, \idealp_n$ are the minimal primes of $R$, and $R_i = R/\idealp_i$ for each such prime, it then follows from this and Lemma \ref{domain} that
\[ \vol_R(\I) =  \lim_{q \to \infty} \( \sum_{i=1}^n  \frac{ \ell ( R_i / I_q R_i)}{q^{\dim R}} \)  =  \sum_{i=1}^n  \( \lim_{q \to \infty} \frac{ \ell ( R_i / I_q R_i)}{q^{\dim R}} \), \] and the theorem follows upon observing that the $i$-th term in the last sum above is zero whenever $\dim(R_i) \neq \dim(R)$.
\end{proof}

\begin{remark} 
\label{volume-as-sum-general: R}
 Proposition \ref{volume-as-sum: P} can be generalized as follows:  Suppose that $\I$ is a $p$-family of $\m$-primary ideals in a local ring $(R, \m)$ of characteristic $p>0$ such that the $R$-module dimension of the nilradical of the completion $\hat{R}$ of $R$ is less than $\dim(R)$.   If $A$ is the quotient of $\hat{R}$ by its nilradical, then as in the proof of Theorem \ref{m-primary limits exist: T}, we know that $\vol_R(\I) = \vol_{\hat{R}}( \I \hat{R}) =  \vol_A(\I A)$, and the latter volume can be described using Proposition~\ref{volume-as-sum: P}.
\end{remark}

\subsection{Consequences of convexity}

In this subsection, we include some results whose proofs rely strongly on convexity-based arguments.

\begin{theorem}
\label{vanishing: T}  Let $(R, \m)$ be an OK local domain of characteristic $p>0$ (e.g., $R$ is an excellent local domain of characteristic $p>0$) and dimension $d$, and let $\I$ be a $p$-family of $\m$-primary ideals in $R$. Then $\vol_R(\I)$ is positive if and only if there exists $q_{\circ}$ a power of $p$ such that  
$I_{qq_{\circ}} \subseteq \m^{[q]}$ for all $q$ a power of $p$.
\end{theorem}

\begin{proof}  Fix a $\ZZ$-linear embedding $\ZZ^d \longhookrightarrow \RR$ induced by a vector $\vv{a} \in \RR^d$, and an OK valuation $\nu: \FF \longonto \ZZ^d$ on the fraction field $\FF$ of $R$ that is OK relative to $R$.  Let $C$ be the cone generated by $\nu(R)$, and let $\Delta$ be $p$-body associated to the $p$-system $\nu(\I)$ in $\nu(R)$.

Observe that if $\vv{w}$ is the least $\nu$-value among the values of some fixed finite generating set for $\m$, then every nonzero element of $\m^{q}$ has value at least $q \vv{w}$.  Thus, if $q_{\circ}$ is such that $I_{qq_{\circ}} \subseteq \m^{[q]}$ for all $q$, then 
\[ q_{\circ} \cdot \nu(I_q) \subseteq \nu(I_{qq_{\circ}}) \subseteq q H, \] 
where $H$ is the halfspace of all points $\vv{u}$ in $\RR^d$ with $\iprod{\vv{u}}{\vv{a}} \geq \iprod{\vv{w}}{\vv{a}}$.   In particular, $(1/q) \cdot \nu(I_q)$, and hence the full $p$-body $\Delta$, lies in $(1/q_{\circ})H$. 
It then follows from Lemma \ref{volume-zero} that $\vol_{\RR^d}(C \setminus \Delta)$ is positive, and Corollary~\ref{OK limit of m-primary ideals exists: E} allows us to conclude that the same is true for $\vol_R(\I)$.

Next, instead suppose that $\vol_R(\I)$, and therefore $\vol_{\RR^d}( C \setminus \Delta)$, is positive, and fix a distinguished point $\vv{v} \in S$ as in Definition \ref{OKvaluation: D}.  If $\mu$ is the minimal number of generators of $\m$, and $\mathscr{H}$ is the halfspace of points $\vv{u}$ in $\RR^d$ with $\iprod{\vv{u}}{\vv{a}} \geq \iprod{\vv{v}}{\vv{a}}$, then Lemma \ref{volume-zero} tells us that for all $q_{\circ}$ large enough, the $p$-body $\Delta$ is contained in $(\mu/q_{\circ}) \cdot \mathscr{H}$.  In particular, $\nu(I_{qq_{\circ}})$ lies in $q\mu \cdot \mathscr{H}$, and our choice of $\vv{v}$ then implies that $I_{qq_{\circ}}  \subseteq R \cap \FF_{ \geq \mu q \vv{v}} \subseteq \m^{\mu q} \subseteq \m^{[q]}$ for all~$q$.
\end{proof}

We record some extensions of Theorem \ref{vanishing: T} below.

\begin{remark}[Positivity of volumes, in general]  Adopt the context of Theorem \ref{vanishing: T}, but rather than assuming that $R$ is an OK local domain, instead suppose that $R$ is reduced and complete.  In this case, Proposition \ref{volume-as-sum: P} tells us that $\vol_R(\I)$ is positive if and only if there exists a minimal prime $\idealp$ of ${R}$ with $\dim(R/\idealp) = \dim(R)$, and such that the volume of $\I$ extended to $R/\idealp$ is positive.  One may then apply Theorem \ref{vanishing: T} to  see that this holds if and only if \[I_{qq_{\circ}}  \subseteq \m^{[q]} + \idealp\] for some $q_{\circ}$ and all $q$ a power of $p$.

Similarly, one may extend Theorem \ref{vanishing: T} to the most general setting in which $\vol_R(\I)$ must exist.  Indeed, if we assume only that the $R$-module dimension of the nilradical of $\hat{R}$ is less than $\dim(R)$, then as in the proof of Theorem \ref{m-primary limits exist: T}, if $A$ is the quotient of $\hat{R}$ by its nilradical, then $\vol_R(\I) = \vol_{\hat{R}}( \I \hat{R}) = \vol_{A}( \I A)$, and the positivity of the later volume was considered above.
\end{remark}

\begin{remark} The recent article \cite{PolstraTucker16} is concerned with sequences of ideals $\I = \{ I_{p^e} \}_{e=1}^{\infty}$ of a local ring $R$ of characteristic $p>0$ with the property that 
\begin{equation}
\tag{$\diamond$}
\label{antipfamily} 
 \exists \text{ an $R$-linear } \phi: F_{\ast} R \to R \text{ for which } \phi(I_{pq})\subseteq I_q \text{ for all $q$}.
 \end{equation}   

One may think of this condition as a dual condition to that of forming a $p$-family, in the sense that it guarantees that subsequent ideals in the sequence are \emph{small} in comparison to previous terms, while the condition of forming a $p$-family guarantees that subsequent ideals in the sequence are  \emph{large} in comparison to previous terms. 

Given a sequence of ideals $\I$ in a complete $F$-finite local domain $(R, \m)$ satisfying condition \eqref{antipfamily}, and such that $\m^{[q]} \subseteq I_q$ for all $q$, it is shown in \cite[Theorem 5.5]{PolstraTucker16} that $\vol_R(\I) = 0$ if and only if $\cap_q I_q = 0$.  Moreover, there exists a nonzero element $c \in R$ such that either of these conditions are equivalent to the condition that there exists $q_{\circ}$ a power of $p$ such that $I_{qq_{\circ}} \subseteq (\m^{[q]}: c)$ for all $q$ a power of $p$.  Contrast this with the situation for $p$-bodies, in which it is possible to have $p$-family $\I$ of  $\m$-primary ideals in $R$ with $\cap_{q} I_q = 0$ but such that $\vol_R(\I)$ is positive.  Indeed, the $p$-family given by $I_q= \langle x^{\lfloor \sqrt{q} \rfloor} \rangle$ in $\kk[x]_{\langle x \rangle}$ is one such example.

% that one has the uniformity condition 
%\[\exists c\in R, Q=p^e\in \NN : \forall q, \qquad I_{Qq} \subseteq (\m^{[q]} : c) \]
%for collections of ideals $I_{\bullet}=\{I_q\}$ satisfying the conditions
%\begin{enumerate}
%\item\label{antipfamily} There exists a map $\phi \in \mathrm{Hom}_R(R^{1/p},R)$ with $\phi(I_{pq}^{1/p})\subseteq I_q$ for all $q$.
%\item\label{intersectioncondition} $\cap_{q} \ I_q =0$.
%\end{enumerate}
%
%
%In particular, if $I_{\bullet}=\{I_q\}$ is a $p$-family of ideals, rather than a sequence satisfying condition~(\ref{antipfamily}), that does satisfy condition~(\ref{intersectioncondition}), the limit $\vol_R(I_\bullet)$ may still be zero; the $p$-family $I_q=(x^{\lfloor \sqrt{q} \rfloor})$ in $\kk[x]_{(x)}$ is an easy counterexample.
\end{remark}

We conclude this section with the Brunn-Minkowski type inequality mentioned in the introduction.

\begin{theorem}\label{BMineq}
Let $I_\bullet$ and $J_\bullet$ be $\idealm$-primary $p$-families in a $d$-dimensional local OK domain $(R,\idealm)$ of characteristic $p>0$. If $L_{\bullet}$ is the $p$-family whose $q$-th term is the product of  $I_q$ and $J_q$, then  
\[ \vol_R(\I)^{1/d}+ \vol_R(J_{\bullet})^{1/d} \geq \vol_R(L_{\bullet})^{1/d}.\]
\end{theorem}
\begin{proof} Let $\nu$ be a valuation that is OK relative to $R$, and let $C$ be the cone in $\RR^d$ generated by $\nu(R)$.  Observe that $\nu(I_q) + \nu(J_q) \subseteq \nu(L_q)$ for every $q$.    Thus, if $\Delta', \Delta''$, and $\Delta$ are the $p$-bodies associated to the $p$-systems $\nu(I_{\bullet}), \nu( J_{\bullet})$, and $\nu(L_{\bullet})$ in $\nu(R)$, respectively, then  $\Delta' + \Delta '' \subseteq \Delta$.  

Therefore, 
\begin{align*}
 \vol_{\RR^d}( C \setminus \Delta)^{1/d} &\leq \vol_{\RR^d}( C \setminus ( \Delta' + \Delta''))^{1/d} \\ &\leq \vol_{\RR^d}( C \setminus \Delta')^{1/d} +  \vol_{\RR^d}( C \setminus \Delta'')^{1/d}
 \end{align*}
where the second inequality follows from Theorem~\ref{brunn-minkowski-convex} and the fact that the $p$-bodies under consideration are $C$-convex. The theorem then follows from Theorem~\ref{OK limit existence: T}.
\end{proof}

\section{Applications and examples}\label{applications}
\newcommand{\sat}{\operatorname{sat}}

We record some applications of the results in the previous sections. 

\subsection{Hilbert-Kunz multiplicity}   The following, originally due to Monsky \cite{Monsky83}, is a corollary of Corollary~\ref{OK limit of m-primary ideals exists: E}.

\begin{corollary}[Monsky] Let $(R,\idealm)$ be a complete local domain of prime characteristic, and $I$ and $\idealm$-primary ideal. Then $e_{HK}(I)$ exists.
\end{corollary}

We note that Monsky's result is more general, in that it holds in an arbitrary local ring of prime characteristic.  However, his argument first reduces to the case of a complete local domain.  

The remainder of the present subsection is dedicated to establishing a new case of the existence of the generalized Hilbert-Kunz multiplicity.   Recall that the generalized Hilbert-Kunz multiplicity of an ideal $I$ of a $d$-dimensional local ring $(R, \m)$ is defined as \[  e_{gHK}(I) = \lim_{q \to \infty} \frac{\ell ( H^0_{\m}(R/I^{[q]}))}{q^{d}}.\]  

We also note that $H^0_{\m}(R/I^{[q]}) = (I^{[q]})^{\operatorname{sat}} / I^{[q]}$, where  $\idealb^{\operatorname{sat}} = (\idealb : \m^{\infty})$ denotes the saturation of an ideal $\idealb$ of $R$.

%\!{Define the "sat" notation}

\begin{proposition}\label{gHK} Let $(R,\idealm)$ be either an analytically irreducible local ring, or a graded domain of characteristic $p>0$, and let $I$ be an ideal of $R$ that is homogeneous in the graded case.
If there exists a positive integer $c$ such that
 \begin{equation}\tag{$\dagger$}\label{eq-intersection-condition}  \idealm^{cq} \cap (I^{[q]})^{\mathrm{sat}}  =  \idealm^{cq} \cap I^{[q]} 
 \end{equation}
  for all $q$ a power of $p$, then $e_{gHK}(I)$ exists.
\end{proposition}
\begin{proof} In the local case this follows immediately from Theorem~\ref{OK limit existence: T}, since completion does not affect the relevant lengths. In the graded case, one may reduce to the local case by localizing at $\idealm$, which again does not affect any relevant lengths.
\end{proof}

\begin{remark}   In his proof of the existence of epsilon multiplicity, Cutkosky made use of \cite[Theorem 3.4]{Swanson}
to show that for every ideal $I$ in a Noetherian local ring $(R, \m)$, there exists a positive integer $c$ such that \[ \m^{cn} \cap (I^n)^{\sat} = \m^{cq} \cap I^n\] for every positive integer $n$.  

In the same way, a version of Swanson's uniformity result \cite[Theorem 3.4]{Swanson} for the Frobenius powers of an ideal $I$ would imply that that ideal satisfies the condition \eqref{eq-intersection-condition} in Proposition \ref{gHK}.   However, this appears to be a subtle and complicated matter (but see \cite{SS} for some positive results). 
\end{remark}

\begin{remark}
The condition \eqref{eq-intersection-condition} on $I$ in Proposition \ref{gHK} is related to a notion of interest in tight closure theory. Namely, recall that an ideal $I$ is said to have the property (LC) if there exists a positive integer $c$ such that
\[ \idealm^{c q} (I^{[q]})^{\mathrm{sat}} \subseteq I^{[q]}\]
 for all $q$. Clearly, the condition  (\ref{eq-intersection-condition}) in Proposition \ref{gHK} implies (LC).  
\end{remark}

\begin{remark}
Vraciu has recently shown in \cite{Vraciu16} that if \emph{every} ideal of $R$ satisfies the condition (LC), then $e_{gHK}(I)$ exists for every ideal $I$ of $R$.  According to Proposition \ref{gHK}, it is only necessary for the given ideal $I$ to satisfy (\ref{eq-intersection-condition}) in order that $e_{gHK}(I)$ exist.
 \end{remark}

We stress that it is not well-understood when a given ideal satisfies the condition (LC), and the same is true for the condition  (\ref{eq-intersection-condition}) in Proposition~\ref{gHK}.  However, as we see in the proof of Corollary \ref{gHK: C} below, the latter condition does hold in the most general setting under which (LC) is known to hold.  In order to establish this, we need the following lemma, which was shown to us by Ilya Smirnov.

\begin{lemma}\label{LC-and-intersection-dim1} Let $(R,\idealm)$ be either a local or graded ring of characteristic $p>0$, and let $I$ be an ideal, homogeneous in the graded case, such that $\dim(R/I)=1$.  If there exists a positive integer $c$ such that
\[ \idealm^{c q} (I^{[q]})^{\mathrm{sat}} \subseteq I^{[q]}\]  for all $q$ a power of $p$, then there exists a positive integer $b$ such that \[ \idealm^{b q} \cap (I^{[q]})^{\mathrm{sat}} = \idealm^{b q} \cap I^{[q]}\] for all $q$ a power of $p$.
\end{lemma}
\begin{proof}
Fix a positive integer $c$ such that $\idealm^{c q} (I^{[q]})^{\mathrm{sat}} \subseteq I^{[q]}$ for all $q$, and an element $x \in R$ that forms a parameter modulo $I$.  If $h \geq c$ is an integer, then the $\m$-primary ideal $(I^{[q]} + x^q)^{h}$ lies in $\m^{cq}$, and hence, multiplies the saturation of $I^{[q]}$ into $I^{[q]}$.   In particular, if $h \geq c$, then
 \begin{equation}
\label{saturation comp: e}
(I^{[q]})^{\sat} = ( I^{[q]} : ( I^{[q]} + x^q)^{\infty}) = ( I^{[q]} : ( I^{[q]} + x^q)^{h})  = ( I^{[q]} : x^{hq})
 \end{equation}
 
Now, we claim that
\[ (I^{[q]})^{\sat} \cap (I^{[q]} + x^{cq}) = I^{[q]}.\]
It is obvious that $I^{[q]}$ is contained in the intersection above.  For the other containment, write an element in this intersection as $f=y+x^{cq} z$ with $y \in I^{[q]}$.  By \eqref{saturation comp: e}, we know that $x^{cq}f=x^{cq}y+x^{2cq}z$ must lie in $I^{[q]}$, so $x^{2cq}z \in I^{[q]}$.  On the other hand, this and  \eqref{saturation comp: e} implies that $x^{cq} z\in I^{[q]}$, which allows us to conclude that $f\in I^{[q]}$.

Finally, fix $a$ such that $\idealm^a\subseteq I+(x^c)$.  If $\mu$ is the minimal number of generators of $\idealm^a$, then $\idealm^{a\mu q} \subseteq (\idealm^a)^{[q]} \subseteq I^{[q]}+(x^{cq})$, so that 
\[ \idealm^{a \mu q} \cap (I^{[q]})^{\mathrm{sat}} = \idealm^{a \mu q} \cap (I^{[q]})^{\sat} \cap (I^{[q]} + (x^{cq})) = \idealm^{a \mu q} \cap I^{[q]}.\]
Thus the constant $b=a\mu$ satisfies the conclusion of the lemma.
\end{proof}

\begin{corollary} 
\label{gHK: C}
Let $R$ be a graded domain of characteristic $p>0$, and $I$ be a homogeneous ideal with $\dim(R/I)=1$. Then $e_{gHK}(I)$ exists.
\end{corollary}
\begin{proof} It was shown by Vraciu~\cite{Vraciu00} that such an ideal $I$ satisfies the condition (LC). The result then follows from Lemma~\ref{LC-and-intersection-dim1} and Proposition~\ref{gHK}.
\end{proof}

\subsection{$F$-signature}

Suppose that $(R, \m)$ is a local ring of characteristic $p>0$, and let $\I$ be the sequence whose $p^e$-th term consists of all $z$ in $R$ such that $\phi(z) \in \m$ for every $\phi$ in $\operatorname{Hom}_R(F^e_{\ast} R, R)$.   As in Example \ref{F-sig family: E}, this sequence is a $p$-family,  and it is well-known that the terms in this sequence are all proper and $\m$-primary whenever $R$ is $F$-pure.  In this case, the volume of this $p$-family equals the $F$-signature of $(R, \m)$.  
More generally, if $x$ is a nonzero element of an $F$-pure local ring $(R, \m)$, and $\lambda$ is a positive real parameter, then the limit \[ \lim_{q \to \infty} \frac{ \ell ( R/ (I_q: x^{\lceil{\lambda q \rceil}}) ) }{q^d} \] is called the $F$-signature of the pair $(R, x^{\lambda})$.  

It is known that both versions of $F$-signature equal zero unless the ambient ring is strongly $F$-regular, in which case  it is a domain.  In the case of a local domain, Lemma \ref{F-sig: L} below shows that both these limits always exist.  We note that Corollary \ref{F-sig: C} is originally due to Tucker \cite{Tucker12}, and Corollary \ref{F-sig pairs: C} to Blickle, Schwede, and Tucker \cite{BST}.

\begin{corollary}[Tucker]
\label{F-sig: C}
The $F$-signature of a local domain of prime characteristic exists.
\end{corollary}

\begin{corollary} 
\label{F-sig pairs: C}
If $x$ is a nonzero element of a local domain $R$ of prime characteristic, then the $F$-signature of the pair $(R, x^{\lambda})$ exists for every positive real parameter $\lambda$.
\end{corollary}

We end this subsection with the lemma referenced above.

\begin{lemma}
\label{F-sig: L}
 Let $(R, \m)$ be a $d$-dimensional local ring of characteristic $p>0$.  Fix a $p$-family of $\m$-primary ideals $\I$ of $R$, a positive real number $\lambda$, and a nonzerodivisor $x$ in $R$.  If the $R$-module dimension of the nilradical of $\hat{R}$ is less than $d$, then the limit
\[ \lim_{q \to \infty}  \frac{ \ell (R / (I_q : x^{\lceil \lambda q \rceil - 1} z ) }{q^d} \] 
exist for every nonzerodivisor $z$ in $R$.   Furthermore, the value of this limit does not depend on the choice of the nonzerodivisor $z$.
\end{lemma}

\begin{proof}
By Remark \ref{uniform multiplier: R}, there exists a positive integer $c$ such that $\m^{cq}$ lies in $I_q$ for all $q$ a power of $p$.  For every such $q$, set \[ L_q = (I_q : x^{\lceil \lambda q \rceil - 1} z ).\] 

As in Example \ref{parameter family: E}, the sequence whose $q$-term is \[ J_q = (I_q : x^{\lceil \lambda q \rceil - 1}) \] is a $p$-family in $R$, and by definition, $J_q \subseteq L_q \subseteq (J_q : z)$ for every $q$.  The Lemma then follows from Corollary \ref{flexibility F-sig of pairs: C} and Corollary~\ref{OK limit of m-primary ideals exists: E}.
\end{proof}

\subsection{Explicit computations in the toric case}

Fix a closed full-dimensional pointed cone $C$ in $\RR^d$ generated by a finite subset of $\QQ^d$.   Let $S$ be the semigroup $C \cap \ZZ^d$,  let $\kk$ be a  field of characteristic $p>0$, and let $R$ be the semigroup ring $\kk[S]$.   $D$ will denote the localization of $R$ at $\idealm$, the maximial ideal generated by the monomials with nonzero exponent, and $\FF$ its fraction field.

Consider the valuation  $\nu: \FF^{\times} \longonto \ZZ^d$ that is OK relative to $D$ described in Example \ref{toric ring is OK: E}.  By construction, the semigroup $\nu(D)$ equals $S$, the cone generated by $S$ is the ambient cone $C$, and the extension of residue fields $\kk \hookrightarrow \kk_{V}$ is an isomorphism.  Moreover, if $U$ is a subset of $S$, and $I$ is the ideal of $D$ generated by the monomials whose exponents lie in $U$, then it is straightforward to verify that
\begin{equation}
\label{image of a monomial ideal: e}
\nu ( I ) = U + S. 
\end{equation}

Finally, in this context,  if $\I$ is a $p$-family of $\idealm$-primary ideals in $D$, and $\Delta$ is the $p$-body associated the $p$-system of semigroup ideals $\nu(\I)$ in $S$, then Corollary~\ref{OK limit of m-primary ideals exists: E} tells us that
\begin{equation}
\label{p-body in toric case: e}
\vol_D(\I) = \lim_{q \to \infty}  \frac{\ell ( D/ I_q )}{q^d} = \vol_{\RR^d} ( C \setminus \Delta).
\end{equation} 

\begin{example}[The Hilbert-Kunz multiplicity of a monomial ideal]
Let $\idealb$ be an $\idealm$-primary ideal of $R$ generated by monomials whose exponents belong to some subset $U$ of $S$,  and let $\I$ be the $p$-family of ideals in $D$ whose $q$-th term is the expansion of $\idealb^{[q]}$ to $D$.  It follows from \eqref{image of a monomial ideal: e} that 
$\nu( I_q) = q U + S$, and Example \ref{staircase: E} then tells us that the $p$-body associated to $\nu(\I)$ is $\Delta = U + C$.  In this case, \eqref{p-body in toric case: e} recovers the formula of  Watanabe  \cite{Watanabe99} and Eto \cite{ Eto01},  which states that the  Hilbert-Kunz multiplicity of the expansion of $\idealb$ to the local ring $D$ equals the Euclidean volume of $C \setminus ( U + C)$.  
\end{example} 

\begin{example}[The $F$-signature of a local toric ring] It is well-known that every finitely generated rational cone is rational polyhedral.  In particular, there exist $\vv{v}_1, \dots, \vv{v}_n$ in $\QQ^d$ such that a point $\vv{u}$ in $\RR^d$ lies in $C$ if and only if $\iprod{\vv{u}}{\vv{v}_i}$ is nonnegative for every $\vv{v}_i$.  We may assume that no proper subset of these vectors defines $C$, and after rescaling each vector by a positive rational number, we may also assume that the coordinates of each $\vv{v}_i$ are relatively prime integers.  With these assumptions,  the vectors $\vv{v}_1, \dots, \vv{v}_n$ are uniquely determined by $C$.  

Let $P$ be the polyhedral set consisting of all points $\vv{u}$ in $\RR^d$ such that $0 \leq \iprod{\vv{u}}{\vv{v}_i}  < 1$ for every $\vv{v}_i$, and let $\I$ be the $p$-family ideals in $D$ that define the $F$-signature of $D$.  It is not difficult to verify (see, e.g., \cite[Lemma 6.6(2)]{VonKorff11}) that each $I_q$ is a monomial ideal, and it is shown in Lemma 3.8 and Lemma 3.11 of \cite{VonKorff11} that a monomial $x^{\vv{u}}$ lies in $I_q$ if and only if $\vv{u}$ lies in $ S \setminus qP$.   In light of this, \eqref{image of a monomial ideal: e} tells us that 
$\nu(I_q) =  S \setminus qP  + S$.   Consequently, 
\[ \Delta_q(S, \nu(\I)) = \left( (1/q) \ZZ^d \cap C \right) \setminus P + C,\] from which it follows that the union $\Delta(S, \nu(\I)) = \cup_{q} \, \Delta_q(S, \nu(\I))$ agrees with $C \setminus P$, up to closure.    
In this case, \eqref{p-body in toric case: e} recovers the results of Watanabe and Yoshida \cite{WatanabeYoshida04} and Von Korff \cite{VonKorff11}, which state that the $F$-signature of $D$ equals the Euclidean volume of $C \setminus P$. 
\end{example}

\section*{Acknowedgements}
The authors thank Ilya Smirnov for showing us Lemma~\ref{LC-and-intersection-dim1}, and Thomas Polstra for his questions that inspired us to prove Theorem~\ref{vanishing: T}. We also thank Karl Schwede, Anurag Singh, Karen Smith, Kevin Tucker, and Emily Witt for helpful conversations.

\bibliographystyle{plain}

\end{document}